\documentclass{siamart1116}

\usepackage{amssymb}

\usepackage{amsmath}
\usepackage{mathrsfs}
\usepackage[noend]{algpseudocode}
\usepackage{algorithmicx,algorithm}
\usepackage{booktabs}
\usepackage{color}
\usepackage{subfigure}
\usepackage{ulem}

\newcommand{\hu}{\hat u}
\newcommand{\hv}{\hat v}

\newtheorem{mylemma}{{Lemma}}[section]
\newtheorem{mytheorem}{{Theorem}}[section]
\newtheorem{myproposition}{{Proposition}}[section]

\newtheorem{mydefinition}{{Definition}}[section]
\newtheorem{remark}{Remark}[section]

\newcommand\m[1]{{\color{black}{#1}}}
\begin{document}



\title{A geometrically consistent trace finite element method for the Laplace-Beltrami eigenvalue  problem\thanks{This work is partially supported by  NSFC grant (No. 11971469) and by the National Key R\&D Program of China under Grant 2018YFB0704304 and Grant 2018YFB0704300.}}


\author{
Song Lu \thanks{LSEC, ICMSEC, NCMIS,  Academy of Mathematics and Systems Science, Chinese Academy of Sciences, Beijing 100190, China;
 University of  Chinese Academy of Sciences, Beijing 100049, China, \email{ lusong@lsec.cc.ac.cn} }
\and Xianmin Xu 
\thanks{
LSEC, ICMSEC, NCMIS,  Academy of Mathematics and Systems Science, Chinese Academy of Sciences, Beijing 100190, China;
 Corresponding author, \email{xmxu@lsec.cc.ac.cn}.
}
}

\date{}
\maketitle
\begin{abstract}
In this paper, we propose a new trace finite element method for the {Laplace-Beltrami} eigenvalue problem. The method is proposed directly on a smooth manifold which is implicitly given by a level-set function and require high order numerical quadrature  on the surface.
A comprehensive analysis for the method is provided. We show that the eigenvalues of the discrete Laplace-Beltrami operator coincide with only part of the eigenvalues of an embedded problem, which further corresponds to the finite eigenvalues for a singular generalized algebraic eigenvalue problem. The finite eigenvalues can be efficiently solved by a rank-completing perturbation algorithm in {\it Hochstenbach et al. SIAM J. Matrix Anal. Appl., 2019} \cite{hochstenbach2019solving}. We  prove the method has
optimal convergence rate.  Numerical experiments verify the theoretical analysis
and  show that the geometric consistency can improve the numerical accuracy significantly.
\end{abstract}







\section{Introduction}
Many problems in applied sciences and engineering can {be modeled} by partial differential equations or eigenvalue problems on surfaces. Typical examples
include diffusion of insoluble surfactant on two-phase flow interfaces\cite{Milliken,Stone,GReusken2011}, flows and phase separation in cell
membranes\cite{arroyo2009relaxation,simons1997functional,ElliotStinner,Novaketal} and shape characterization in image processing\cite{bertalmio2000variational,reuter2006laplace,xu2004discrete}, etc.
In particular, the Laplace-Beltrami eigenvalue problem has important applications to characterize  the shape of a surface, and is referred as the shape-DNA in literature\cite{reuter2006laplace,nasikun2018fast}. The spectra of the Laplace-Beltrami operator is also an important topic in \m{geometry} \cite{mckean1967curvature,gordon1992isospectral,buser2010geometry}, starting from the well-known Weyl theorem on asymptotic growth of eigenvalues\cite{weyl1912asymptotische}. It is found that  the spectra of the Laplace-Beltrami \m{operator is} isometric invariant and many important geometric property can be computed thereby \cite{craioveanu2013old}.
The nice property is also crucial in many applications in inverse problems\cite{gordon1992one}.

Solving partial differential equations on surfaces has arisen much interest in the community of numerical analysis recently. Various numerical methods have been developed, \m{including} the finite difference methods \cite{bertalmio2000framework,XuZh,ruuth2008simple,leung2011grid,macdonald2010implicit,beale2020solving}, finite element methods\cite{dziuk1988finite,dziuk2007finite,dziuk2007surface,demlow2009higher,deckelnick2010h,kovacs2018high},
meshless methods \cite{lehto2017radial,li2016convergent,liang2013solving} and many others \cite{du2003voronoi,dede2015isogeometric}. More information can be found in the recent review papers \cite{bonito2020finite,dziuk2013finite,olshanskii2017trace} and the references therein.
In this work, we will focus on the trace finite element method, which was first developed by Olshanskii, Reusken and Grade in \cite{olshanskii2009finite}
for {Laplace-Beltrami} equations on stationary surfaces. Assume that the surface is embedded in a bulk domain which is triangulated and incorporated  with some standard finite element spaces. Then the discrete surface is constructed by {piecewise} planar approximations of the smooth surface. The key idea of the trace FEM is to use the traces of the bulk finite element functions on the discrete
surface to construct the finite element spaces on surfaces.
The method has optimal convergence even though
the resulted linear {algebraic} problem might be {degenerate}.
The matrix property is further studied in \cite{olshanskii2010finite}. To overcome  the degeneracy of the system, some stabilization techniques have been developed in
\cite{burman2015stabilized,burman2016full,larson2020stabilization}, where the method is called the cut-FEM. The trace FEM has also been further
developed in several directions, like to consider higher order finite element approximations \cite{reusken2015analysis,lehrenfeld2016high,grande2016higher,grande2018analysis,larson2020stabilization}, to use {discontinuous} Galerkin approximations\cite{burman2016cutb} and
 adaptive finite element approximations\cite{demlow2007adaptive,chernyshenko2015adaptive}, and  to solve problems on evolving surfaces\cite{olshanskii2014eulerian,olshanskii2017trace,lehrenfeld2018stabilized}, etc.
Recently, the method has also been applied to study the Navier-Stokes equations and the phase-field equations
on surfaces\cite{olshanskii2021inf,yushutin2020numerical,olshanskii2021finite}.

In comparison with partial differential equations on surfaces,
 the numerical study on the \m{Laplace-Beltrami} eigenvalue problem is relatively few in the literature. Previous methods include  the closest point method \cite{macdonald2011solving}, the parametrization method \cite{glowinski2008computing}, the discontinuous Galerkin method \cite{dong2020discontinuous}, etc.
We would like to use the trace FEM to solve the Laplace-Beltrami eigenvalue problem in this work.
The difficulties to approximate the LB eigenvalue problem by the trace finite element method come from two aspects.
Firstly, the number of freedoms is usually  larger than the dimension of the trace FEM spaces. In this case, the definition
of the discrete eigenvalue problems is not clear since there will be many false {solutions} for the discrete (generalized)
eigenvalue problems. Secondly, the discretization of the surface will introduce some geometric errors which affect the accuracy {of the} eigenvalue problems. The degeneracy of the trace finite element method on discrete surfaces may cause severe
problems in calculating true eigenvalues.
The geometric inconsistency errors {appear} in almost
all the previous numerical methods \cite{bonito2020finite}. A few methods which are geometrically consistent are {those} by isogeometric analysis\cite{dede2015isogeometric}, the method with exact geometric description\cite{gfrerer2018high} and the recent developed intrinsic finite element method \cite{bachini2021intrinsic}.

In this paper, we develop a new geometrically consistent trace FEM method for the {Laplace-Beltrami} eigenvalue problem. The method is based on the high order quadrature directly proposed  on curved surfaces and also  on the approximation of the embedded problems. We carefully analyse the discrete embedded problems and
give the conditions under which the embedded problem is equivalent to the original ones, where the geometric consistency can play an important role.
We also prove that the true eigenvalues of the trace finite element approximation coincide  with the finite eigenvalues of a singular generalized algebraic eigenvalue problem\cite{Demmel2000}. This {enables} us to utilize a rank-completing perturbation algorithm in  \cite{hochstenbach2019solving} to solve the  possibly singular discrete {eigenvalue} problem.
We provide a detailed error analysis of the method and prove that the optimal convergence rate can be achieved.
Numerical examples are given for both the Laplace-Beltrami equation and the \m{corresponding} eigenvalue problems.  It is found that the geometrical consistency can improve the accuracy dramatically in both cases. For the Laplace-Beltrami  eigenvalue problems,
we show that the geometrically consistent method can produce much \m{fewer} false
eigenvalues than the original trace finite element method.
In addition, we would like to remark that the method can be directly extended
to higher order finite element approximations.

The rest of the paper is organized as follows.
In section 2, we briefly introduce the continuous model problems.
In section 3, we consider the discretization of the continuous problems by the trace finite element method.
The implementation details of the method are presented in section 4. 
In section 5, we conduct a rigorous error analysis of the method.
Numerical experiments are illustrated in section 6 to verify the theoretical results and to show the \m{efficiency} of the method.
Some \m{concluding} remarks are given in Section 7.

\section{The model problems}
In this section, we  briefly introduce two model problems corresponding to the \m{Laplace-Beltrami} operators on a generally smooth surface.
The first is a {Laplace-Beltrami} type equation and the second is the corresponding eigenvalue problem.
Although we are mainly interested in the eigenvalue problem, the analysis for the Laplace-Beltrami equation will be the basis for that of the eigenvalue problem.
\subsection{The Laplace-Beltrami equation}
Let $\Gamma$ be a closed smooth surface contained in a domain $\Omega \subset \mathbb{R}^{3}$ and $f\in L^{2}(\Gamma)$. A Laplace-Beltrami type equation with a zeroth-order term on $\Gamma$ is given as
\begin{equation}\label{m0}
  -\Delta_{\Gamma} u +c u =f \quad \text { on } \Gamma.
\end{equation}
Here 
$\Delta_{\Gamma}$ is the Laplace-Beltrami operator\cite{bonito2020finite} {and $c$ is a real constant}.
Denote by $\nabla_{\Gamma}$ the surface gradient operator  and by $H^1(\Gamma)$ the standard Sobolev space defined
on $\Gamma$ \cite{hebey1996sobolev,bonito2020finite}. Denote by $(\cdot,\cdot)_{\Gamma}$ the $L^2$-inner product on $\Gamma$.
The weak form of the problem is to find a function $u \in H^{1}(\Gamma)$ such that
\begin{equation}\label{weak_form}
a(u,v) + (cu,v)_{\Gamma} = ( f,v )_{\Gamma} \quad \text{for all} \quad v \in H^{1}(\Gamma),
\end{equation}
where the bilinear form $a(\cdot,\cdot)$ is
\begin{equation}\label{auvbuv}
  a(u,v):=\int_{\Gamma} \nabla_{\Gamma} u \cdot \nabla_{\Gamma} v \mathrm{d} {s},\qquad
\end{equation}
and
$$(cu,v )_{\Gamma} = \int_{\Gamma} cuv \mathrm{d} {s}, \quad (f,v )_{\Gamma} = \int_{\Gamma} fv \mathrm{d} {s}.$$
When $c>0$, the  Lax-Milgram theorem  implies
the problem (\ref{weak_form}) has a unique solution.
When $c\leq 0$, the wellposedness of the problem is closely related
to the {Laplace-Beltrami} eigenvalue problem described below. In this case,
the problem is sometimes referred as the Helmholtz-Betrami equation \cite{burman2020stable}. 
\subsection{The Laplace-Beltrami eigenvalue problem}
On the closed smooth surface $\Gamma$, the standard Laplace-Beltrami eigenvalue problem is
to find a pair $(\lambda, u) \in\left(\mathbb{R}^{}, H^{2}(\Gamma)\right)$ where $u\neq 0$,  such that
\begin{equation}\label{eig_problem}
-\Delta_{\Gamma} u=\lambda u.
\end{equation}
One can easily verify that the problem has a trivial eigenvalue $0$, and the corresponding eigenfunction is $u\equiv \text{constant}$. All other eigenvalues
are positive.
Furthermore, on a smooth closed compact orientable manifold, the zero eigenvalue of the Laplace-Beltrami operator has multiplicity 1.
Therefore, when the coefficient $c=0$ in \eqref{weak_form}, one needs a consistency condition that $( f, 1)_{\Gamma} =0$
by Fredholm's alternative and the solution of \eqref{weak_form} is unique up to a constant.  Similar arguments can be done for the case $c<0$ where $-c$ is an eigenvalue of the Laplace-Beltrami operator.

Suppose that the eigenvalues of the problem~\eqref{eig_problem} are ordered as
$$0=\lambda_{1} \leq \lambda_{2} \leq \cdots \leq \lambda_{n} \leq \cdots,$$
and the corresponding eigenfunctions $u_{i}$ satisfying $\left\|u_{i}\right\|_{L^{2}(\Gamma)}=1$.
The weak form of (\ref{eig_problem}) is to find a pair $(\lambda, u) \in\left(\mathbb{R}^{}, H^{1}(\Gamma)\right)$ with $\|u\|_{L^{2}(\Gamma)}=1$ such that
\begin{equation}\label{eig_problem_weak}
a(u,v) = \lambda (u,v)_{\Gamma}.
\end{equation}
\m{It is obvious that the eigenfunction $u$ satisfies $\int_\Gamma u = 0$ when $\lambda \neq 0$.}
\subsection{Extensions to a neighbourhood region}
In  this subsection, we consider extensions of the above two problems in a neighbouring bulk region of $\Gamma$ which are intuitive for us to design proper trace finite element methods for the Laplace-Beltrami eigenvalue problem.

Denote by $d(x)$ a signed distance function to $\Gamma$.
Define a narrow band neighbouring region of $\Gamma$ with width $2\delta$ as,
\begin{equation}
\mathcal{N}_\delta=\left\{\mathrm{x} \in \mathbb{R}^{3} \mid \left|d(\mathrm{x})\right|<\delta\right\}.
\end{equation}
Then we introduce a Sobolev space in $\mathcal{N}_\delta$ composed of functions with trace on $\Gamma$ belonging to $H^1(\Gamma)$
as,
\begin{equation}\label{extension_space}
\hat{H} = \left\{ \hat{v} \in H^{\frac{3}{2}}(\mathcal{N}_{\delta})\mid \exists v\in H^1(\Gamma): \hat{v} |_\Gamma = v\right\}.
\end{equation}
Then the (ill-posed) embedded problem corresponding to (\ref{weak_form}) is
to find a function $\hat{u} \in \hat{H}$,
\begin{equation}\label{m0_ex}
  a(\hu,\hv)+(c\hu,\hv)_{\Gamma}=(f,\hv)_{\Gamma}, \quad \forall \hv\in\hat{H}.
\end{equation}
Similarly, the embedded problem corresponding to (\ref{eig_problem}) is to find
a pair $(\lambda, \hu) \in (\mathbb{R}, \hat{H})$ with $\left\| \hu \right\|_{L^2(\mathcal{N}_\delta)}=1$ such that
\begin{equation}\label{eig_problem_ex}
a(\hu,\hv)=\lambda (\hu,\hv)_{\Gamma}, \qquad \forall \hv\in \hat{H}.
\end{equation}

The following results are trivial for the Laplace-Beltrami problem.  For a solution $u$ of \eqref{weak_form}, any extension of $u$ in $\mathcal{N}_\delta$ is a solution of \eqref{m0_ex} whenever
it is in $\hat{H}$. Since the extension of a function $u\in H^1(\Gamma)$ to $\hat{H}$ are not unique,
the problem \eqref{m0_ex} is ill-posed.
Nevertheless, if $\hu \in \hat{H}$ is a  solution of \eqref{m0_ex},
 its trace on $\Gamma$ is always a solution of \eqref{weak_form}. Therefore, once the problem~\eqref{weak_form}
has a unique solution,  all the solutions of \eqref{m0_ex} corresponds to the same trace on $\Gamma$.
This fact is the basis for the original trace finite element method \cite{olshanskii2009finite}.

The relation between the problem \eqref{eig_problem_weak} and the embedded eigenvalue problem (\ref{eig_problem_ex})  is more  tricky.
Firstly, for a solution $(\lambda, u)$ of \eqref{eig_problem_weak}, we could find a function $\hu \in \hat{H}$
such that $\|\hu \|_{L^2(\mathcal{N}_\delta)}=1$ and $\hu|_{\Gamma}=u$. Then $(\lambda, \hu)$ is the solution
of \eqref{eig_problem_ex}. The extension is not unique, anologously to the Laplace-Beltrami equation.
\m{Secondly}, there exists more trouble if one would like to relate the solution of the problem \eqref{eig_problem_ex}
to that of \eqref{eig_problem_weak}. Actually, notice that both the bilinear forms $a(\cdot,\cdot)$ and $(\cdot,\cdot)_{\Gamma}$
 are defined only on  $\Gamma$. If we consider a function $\hu \in \hat{H}$ such that $\|\hu \|_{L^2(\mathcal{N}_\delta)}=1$ and
$\hu|_{\Gamma}\equiv 0$,
 then any $\lambda \in \mathbb{R}$ can be the eigenvalue of \eqref{eig_problem_ex}, while
 it may not be the eigenvalue of \eqref{eig_problem_weak}.  This indicates
  the difficulty of direct application of the standard trace finite element method to the   Laplace-Beltrami eigenvalue problem.

To overcome this degeneracy of the embedded eigenvalue problem,
 we introduce a definition for the ``true eigenvalues"  of \eqref{eig_problem_ex}.
 \begin{mydefinition}[true eigenvalue]
 The true eigenvalues of \eqref{eig_problem_ex} are those corresponding to
 at least one eigenfunction $\hu$, such that $\hu|_{\Gamma}\neq 0$.
 \end{mydefinition}
 With this definition, we can easily see that any true eigenvalue of \eqref{eig_problem_ex}  is also
 an eigenvalue of \eqref{eig_problem_weak}. We will use this fact to design
 a trace finite element method for the  \m{Laplace-Beltrami} eigenvalue problem.

\section{The trace finite element method}
\label{ExTraceFEM}
In this section, we will introduce a trace finite element method proposed on the smooth surface, which is different
from previous versions of the method
\cite{olshanskii2009finite,olshanskii2010finite,grande2014eulerian,burman2015stabilized,
reusken2015analysis,lehrenfeld2016high,grande2016higher,olshanskii2017trace,grande2018analysis,bonito2020finite}.
The implementation and  error analysis of the method will be given in following sections.
\subsection{Notations}
Suppose $\Gamma$ is embedded in a bounded domain $\Omega \subset \mathbb{R}^3$.
Let $\mathcal{T}_h$ be a shape-regular tetrahedral partition of $\Omega$. Denote by $h_T$ the diameter of an element $T \in \mathcal T_h$ and $h:=\max_{T \in \mathcal T_h} h_T$. The intersection of an element $T$  with $\Gamma$ is  $F_T :={T \cap \Gamma}$. Define
\begin{equation}
  \mathcal T_h^\Gamma := \{T \in \mathcal T_h\ |\ \hbox{meas}_2(F_T)\neq 0 \}.
\end{equation}
Here we assume $F_T$ is in the interior of $T$ for simplicity. Otherwise, if $F_T$ is the boundary of two neighbouring tetrahedrals, we keep only one of them in $\mathcal{T}_h^\Gamma$.

 The simplexes intersecting with $\Gamma$  form a tubular region,
 \begin{equation}\omega_{h}:=\cup_{T \in \mathcal{T}_h^{\Gamma}} {T}\end{equation}
 For a fixed $\delta$, when $h$ is small enough, we have $\omega_h \subset \mathcal{N}_{\delta}$.
For any $T \in \mathcal T_h$, $P_k(T)$ stands for the set of $k$-th order polynomials on $T$. The standard $k$-th order Lagrangian finite
element \m{space} on $\omega_h$ is defined as
\begin{equation}\label{Wh}
W_h:=\left\{\hat{v}_h \in C(\omega_h)\ \big|\ \hat{v}_h |_{T} \in P_{k}(T), \forall T \in \mathcal{T}_h^\Gamma\right\}.
\end{equation}
The traces on $\Gamma$ of the functions in $W_h$ form a finite dimensional linear space
\begin{equation}\label{Pk_Gamma}
V_h: = \left\{v_h \in C\left(\Gamma\right)\ \big| \ \exists \hat{v}_h \in W_h \  s.t. \ v_h =\hat{v}_h |_{\Gamma_{}}\right\}.
\end{equation}
One can easily see that $V_h$ is a subspace of $H^1(\Gamma)$. However, the function in $V_h$ may not
be a polynomial in parametric coordinates for  a general curved surface $\Gamma$.

We define the following extensions of a function $v_h\in V_h$ with reference to $\omega_h$, that is
\begin{equation}\label{extension_space_h}
\operatorname{E}(v_h) = \left\{ \hv_h \in W_h  \mid \hv_h |_\Gamma = v_h\right\}.
\end{equation}
Accordingly, we define a restriction operator, for any $ \hv_h \in \operatorname{E}(v_h)$,
\begin{equation}\label{restriction_space_h}
\operatorname{R}( \hv_h) = v_h.
\end{equation}

\subsection{The trace finite element method}
The standard Galerkin approximation of the Laplace-Beltrami equation (\ref{weak_form}) reads:  to find
$u_{h} \in  V_{h}$ such that
\begin{equation}\label{weak_form_apxm}
a(u_h,v_h) +(c u_h, v_h)_{\Gamma} = ( f,v_h )_{\Gamma}, \quad \text{for all} \quad v_h \in  V_{h}.
\end{equation}
Similar to the continuous problem~\eqref{weak_form}, the well-posed of
the problem can be proved by using the Lax-Milgram theorem when $c>0$.

Notice the problem~\eqref{weak_form_apxm} cannot be implemented directly since
the basis of $V_h$ is not known explicitly. In practice, we actually solve the following problem:
to  find a function \m{$\hat{u}_h \in W_h$} such that
\begin{equation}\label{weak_form_apxm_ex}
a(\hu_h,\hv_h) +(c \hu_h, \hv_h)_{\Gamma} = ( f, \hv_h )_{\Gamma}, \quad \text{for all} \quad \hv_h \in  W_{h}.
\end{equation}
It can be seen as an approximation of an embedded problem \eqref{m0_ex} which is defined in a  subdomain $\omega_h$ instead of $\mathcal{N}_{\delta}$.
Similar to the continuous problems, we know that
if $u_h$ is a solution of (\ref{weak_form_apxm}), then any $\hat{u}_h \in \operatorname E(u_h)$ is
a solution of (\ref{weak_form_apxm_ex}). If $\hu_h$ is a solution of (\ref{weak_form_apxm_ex}), we also know that
$\operatorname R(\hu_h)$ must be a solution of (\ref{weak_form_apxm}).

For the eigenvalue problem (\ref{eig_problem_weak}), its Galerkin approximation on $V_h$ is to
find  pairs $\left(\lambda_{h}^{}, u_{h}^{}\right) \in \left(\mathbb{R}, V_{h}^{}\right)$ with $\left\|u_{h}^{} \right\|_{L^2(\Gamma)}=1$ such that
\begin{equation}\label{eig_problem_discrete}
a(u_h,v_h) = \lambda_h (u_h,v_h)_{\Gamma} \quad \text{for all} \quad v_h \in  V_{h}.
\end{equation}
Once again, the problem cannot be implemented directly. In  practise, we will solve an alternative problem as follows.
Find pairs $\left(\lambda_{h}^{}, \hu_{h}^{}\right) \in \left(\mathbb{R}, W_{h}^{}\right)$ with {$\left\|\hu_{h}^{} \right\|_{L^2(\omega_h)}=1$} such that
\begin{equation}\label{eig_problem_discrete_ex}
a(\hu_h,\hv_h) = \lambda_h (\hu_h,\hv_h)_{\Gamma} \quad \text{for all} \quad \hv_h \in  W_{h}.
\end{equation}
It can be seen as a discrete form of the extension problem (\ref{eig_problem_ex}) in a subdomain $\omega_h$ of $\mathcal{N}_{\delta}$.
Similar to the continuous Laplace-Beltrami eigenvalue problems, the relation between \eqref{eig_problem_discrete} and \eqref{eig_problem_discrete_ex} is tricky and  will be clarified below.


\subsection{Analysis of the discrete embedded problems}\label{discrete_embedded}
 In general the embedded problems (\ref{weak_form_apxm_ex}) and (\ref{eig_problem_discrete}) may not be well-posed, that is similar to the continuous problems. However, since the extensions presented in $\operatorname{E}(v_h)$ (for a function $v_h\in V_h$) is not arbitrary, it is possible
  that the embedded problems are well defined under some conditions.
To show this, we first define the kernel space of the Laplace-Beltrami operator in $W_h$ as follows.
\begin{mydefinition}[Discrete kernel space]
\begin{equation}\label{ker}
  \operatorname{Ker}_h(\Delta_{\Gamma}):=\left\{\hat{w}_h \in W_h \mid a(\hat{w}_h,\hv_h)=0,\forall \hv_h \in W_h \right\}.
\end{equation}
\end{mydefinition}
We have the following lemma.
\begin{mylemma}\label{con_wellpose}
The following conditions are equivalent.
\begin{enumerate}
\item[(i).]   $  \operatorname{Ker}_h(\Delta_{\Gamma})= \mathrm{span}\left\{1\right\}$.
\item[(ii)]  $\operatorname{E}({0}) = \{0\}$.
\item[(iii)]
$\Gamma$ {is not a part of} the zero level set of any non-zero finite element function in $W_h$.
\item[(iv)]  $\operatorname{dim} (W_h)  =  \operatorname{dim} (V_h)$.
\end{enumerate}
\end{mylemma}
\begin{proof}
   { We prove the lemma by the method of contradiction.}

  $(i)\Rightarrow (ii)$:
  \m{If there exists a $\hv_h \in W_h$, $0\neq \hv_h\in \operatorname E(0)$, then $\hv_h\in \operatorname{Ker}_h(\Delta_{\Gamma})$ but $\hv_h \notin \mathrm{span}\left\{1\right\}$, which contradicts (i).}

  $(ii)\Rightarrow (iii)$:
  If there is a non-zero function $\hat w_h \in W_h$ such that \m{$\Gamma \subseteq \{x|\hat w_h(x)=0\}$}, then for any $ x \in \Gamma, \hat w_h(x)=0$,
  which means $ \hat w_h|_\Gamma = 0$ so that $\hat w_h \in \operatorname E(0)$. This is contradictory to (ii).

  $(iii)\Rightarrow (iv)$: We easily know that $\operatorname{dim} V_h \leq \operatorname{dim} W_h$. We need only to show that
  the inequality does not hold.
  Assuming that $\{\phi_i\}$ is a set of basis functions of $W_h$ space. If $\operatorname{dim}W_h>\operatorname{dim}V_h$,
  then there is $\mathbf{\alpha} = \{\alpha_i \}\neq \mathbf{0}$ such that $\sum\alpha_i \phi_i |_\Gamma = 0$. Let $\hat w_h = \sum \alpha_i \phi_i \in W_h$, it is easy to know that $\hat w_h\neq 0 $ and \m{its zero level set includes $\Gamma$}, which contradicts (iii).

  $(iv)\Rightarrow (i)$:  Let $\{\phi_i\}$ are the standard Lagrangian finite element basis in $W_h$.
  If $\operatorname{Ker}_h(\Delta_{\Gamma})\neq \text{span}\{1\}$, then there exists a non-constant function $\hat w_h=\sum \alpha_i \phi_i,
  \mathbf{\alpha} = \{\alpha_i\}\notin \text{span}\{\mathbf{1}\}$ such that $a(\hat w_h,\hat v_h)= 0$ for all \m{$\hv_h \in W_h$}. Here $\mathbf{1}$ represents
  a vector with every component equal to $1$.
  Choose $\hv_h=\hat w_h$ in the equation, we easily know that
  $\hat{w}_h=const$ on $\Gamma$. Denote the constant as $c_0$. We can get $\operatorname R (\hat w_h-c_0) = \sum (\alpha_i-c_0) \phi_i|_{\Gamma} = 0$, which implies that $\{\phi_i |_\Gamma\}$ is linearly dependent and $\operatorname{dim} V_h <\operatorname{dim} W_h$.
\end{proof}

Using the lemma, we easily have the following result on the existence of the solutions of the embedded problems.
\begin{myproposition}
  Under the conditions of Lemma \ref{con_wellpose}, the following facts hold:
  \begin{enumerate}
\item[(i).]  There exists a unique solution for the problem (\ref{weak_form_apxm_ex}) when $c>0$.
\item[(ii).] The problems \eqref{eig_problem_discrete} and \eqref{eig_problem_discrete_ex} have the same eigenvalues.
  \end{enumerate}
\end{myproposition}

When the conditions of Lemma~\ref{con_wellpose} do not hold, we know that $\operatorname{dim} V_h <\operatorname{dim} W_h$.
We can choose a non-zero function $\hat w_h\in W_h$, such that $\operatorname R(\hat w_h)=0$ on $\Gamma$.
Similar to the continuous problem, we have $a(\hat{w}_h,v_h)=(\hat{w}_h,v_h)_{\Gamma}=0$, for all $v_h\in W_h$. Then
for the embedded problem~\eqref{eig_problem_discrete_ex}, any number $\lambda\in \mathbb{R}$ is an eigenvalue.
Since we are interested only in the eigenvalues which are the same as that of \eqref{eig_problem_discrete}, we
introduce the following definition for  ``true eigenvalues".
 \begin{mydefinition}The true eigenvalues of \eqref{eig_problem_discrete_ex} are those corresponding to
 at least one eigenfunction $\hu_h$, such that $\hu_h|_{\Gamma}\neq 0$.
 \end{mydefinition}
 We will show how to compute the true eigenvalues of \eqref{eig_problem_discrete_ex} in next section.

\section{Implementation of the method}
In this section, we describe the numerical implementation on solving the discrete problems~\eqref{weak_form_apxm_ex} and \eqref{eig_problem_discrete_ex}.
We  discuss mainly two issues. The first is to do numerical integration on a smooth manifold $\Gamma$ or piecewisely on $F_{T}=T\cap\Gamma$
for each $T\in\mathcal{T}_{h}^{\Gamma}$. The second is the algebraic problem to compute the true eigenvalues of~\eqref{eig_problem_discrete_ex}.
\subsection{Quadrature on curved surfaces}
\label{high-order numerical quadrature}
Let $T \in \mathcal{T}_h$ be a tetrahedron in $\Omega \subset \mathbb{R}^3$, and $\Gamma$ be a smooth surface intersecting $T$, as shown in Figure \ref{localSurface}. $\Gamma$ is given implicitly by a level set function $\Gamma=\{\mathrm{x} \in \Omega \mid \varphi(\mathrm{x})=0\}$. Suppose $u(x):\Gamma \rightarrow \mathbb{R}$ is a continuous function. We use the method proposed in \cite{cui2020high} to
numerically calculate
\begin{equation}I=\int_{F_T} u({x}) \mathrm{d} S.\end{equation}
\begin{figure}[H]
  \centering
  \includegraphics[width=6cm]{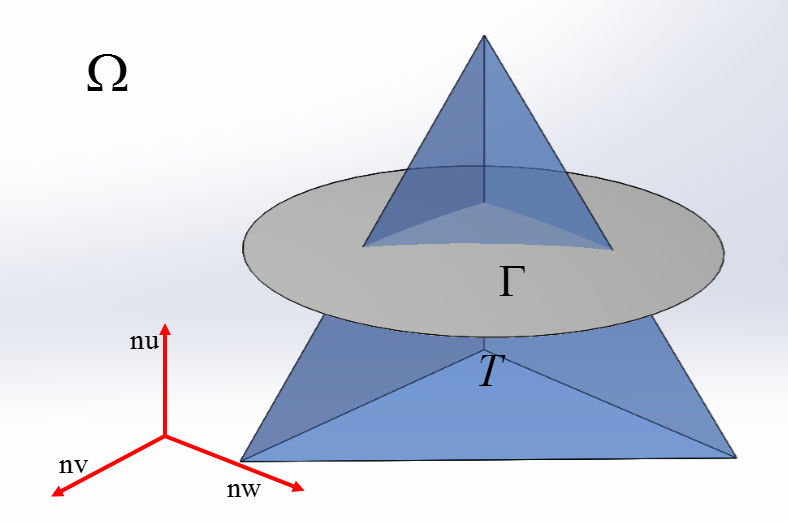}\\
  \caption{Intersection of surface and tetrahedron}\label{localSurface}
\end{figure}
We describe the method briefly as follows. As shown in Figure~\ref{localSurface}, suppose we choose an appropriate rectangular coordinate system $\{x_0, \{{\boldsymbol {\mathrm{nu}},\boldsymbol {\mathrm{nv}},\boldsymbol {\mathrm{nw}}}\}\}$ so that $T$ is contained in a rectangular parallelepiped unit, i.e.,
\begin{equation}
  T \subset\left\{\mathrm{x}_{0}+r \mathbf{nu}+s \mathbf{nv}+t \mathbf{nw} \mid r \in(0, a), s \in(0, b), t \in(0, c)\right\}.
\end{equation}
Then we directly use a projection method to calculate the surface integral in
the parametric domain, i.e.,
\begin{equation}\label{int_u}
I =\int_{F_T} u(\mathrm{x}) \mathrm{d} S=\int_{0}^{c} \int_{0}^{b} \tilde{g}(s, t) \mathrm{d} s \mathrm{d} t=\int_{0}^{c} \m{\tilde{h}(t) \mathrm{d}t,} \end{equation}
where
\begin{equation}\begin{array}{l}
\tilde{g}(s, t):=\left\{\begin{array}{ll}
u\left(r_{0}, s, t\right) & \frac{\left|\nabla \varphi \left(\mathrm{x}\left(r_{0}, s, t\right)\right)\right|}{\left|\mathbf{n} \mathbf{u} \cdot \nabla \varphi \left(\mathrm{x}\left(r_{0}, s, t\right)\right)\right|}, \quad \text { if } \exists r_{0}, \text { s.t. } \mathrm{x}\left(r_{0}, s, t\right) \in F_T,\\
0, & \text { otherwise;}
\end{array}\right. \\
\tilde{h}(t):=\int_{0}^{b} \tilde{g}(s, t) \mathrm{d} s.
\end{array}\end{equation}
\m{
Although the approach looks straightforward, there are difficulties in implementation. First, the integrand is discontinuous and we need search for the discontinuity points
in each interval of integration. Second, the integrand might be singular if the normal
$\mathbf{nu}$ is not chosen properly so that it is tangential to $F_T$ at some point.
In this case, we need detect the singular points in the integration and change coordinates to avoid them. More details are referred to \cite{cui2020high}.
}
Other approaches for high order
quadratures on curved surfaces can be found  in \cite{muller2013highly,saye2015high}.

\subsection{The algebraic problem}
Suppose the finite element basis functions of $W_h$ are given by $\{\phi_1, \phi_2, \cdots, \phi_M\}$.
We express the solution of (\ref{weak_form_apxm_ex}) 
as $\hu_h = \sum_{i=1}^{M} u_i \phi_i$.
Let $\mathbf{u} = (u_1 ,u_2,\cdots,u_M)$.
When $c$ is a constant, the algebraic system of the discrete Laplace-Beltrami problem (\ref{weak_form_apxm_ex}) can be written as,
\begin{equation}\label{ABx}
  (A+ c B)\mathbf{u}=\mathbf{f},
\end{equation}
where the elements of the matrices $A$ and $B$ are given by
\begin{equation*}
a_{ij}=a(\phi_i,\phi_j), \quad
b_{ij}=(\phi_i,\phi_j)_{\Gamma}, \qquad i, j=1, \cdots,M,
\end{equation*}
and $\mathbf{f} = \{\left(f,\phi_i\right)_{\Gamma}, i=1,\cdots,M\}$. Similarly, the algebraic problem corresponding to (\ref{eig_problem_discrete_ex}) can be written as,
\begin{equation}\label{AxLBx}
  A\mathbf{u}=\lambda B \mathbf{u}.
\end{equation}

When the conditions of Lemma \ref{con_wellpose} are not satisfied, neither the algebraic problem (\ref{ABx}) nor the generalized eigenvalue problem (\ref{AxLBx}) are well defined.
For the algebraic problems, we introduce a new equivalent condition of Lemma \ref{con_wellpose}.
\begin{mylemma}\label{con_wellpose_matrix}
  The condition of Lemma \ref{con_wellpose} holds if and only if $B$ is nonsingular.
\end{mylemma}
\begin{proof}
  Notice the equivalence between the condition (iv) in Lemma \ref{con_wellpose} and the definition of $B$, the conclusion can be drawn immediately.
\end{proof}

When the condition of Lemma \ref{con_wellpose_matrix} holds, we easily see that
$\operatorname{rank}(B) = \operatorname{dim}(W_h) = \operatorname{rank}(A)+1=\m{\operatorname{rank}(A+cB)}$ when $c$ is not an
 eigenvalue of \eqref{AxLBx}.
In this case, the algebra problem $(A+cB)\mathbf{u}=\mathbf{f}$ has a unique solution.
Meanwhile, since the matrix $B$ is invertible,  $A\mathbf{u}=\lambda B \mathbf{u}$ can be understand in a usual way,
  \begin{equation}\label{invBAx}
    B^{-1}A\mathbf{u} = \lambda \mathbf{u}.
  \end{equation}

When the condition of Lemma \ref{con_wellpose_matrix} is not satisfied,
the algebra problem $(A+cB)\mathbf{u}=\mathbf{f}$ may have multiple solutions
even when $c>0$. However,  all of the solutions correspond to the same trace \m{$\sum_{i=1}^{M}u_i\phi_i |_\Gamma$} on $\Gamma$. Thus we can still obtain an approximate solution for the Laplace-Beltrami equation
whenever we find a solution of~\eqref{ABx}.
When $B$ is singular, there is more trouble for the eigenvalue problem, as discussed in Section~\ref{discrete_embedded}.
The problem $A\mathbf{u}=\lambda B \mathbf{u}$ is not well defined since both $A$ and $B$ are singular.
One can show that there exists a non-zero vector $\mathbf{u}$ such that  $A\mathbf{u}=\mathbf{0}=B\mathbf{u}$. In this case,
any $\lambda\in \mathbb{R}$ satisfies the equation.
\m{Moreover, we cannot get the multiplicity of a specific eigenvalue correctly.}
To focus on the eigenvalues we are interested in, we need study the ``true eigenvalues" of the generalized eigenvalue problem~\eqref{AxLBx}.

We  recall some standard definitions for the {\it finite eigenvalues} of singular generalized eigenvalue problems(which means that $\mathrm{det}(A-\lambda B) \equiv 0$ for all $\lambda\in \mathbb{C}$) \cite{Demmel2000}.
We first introduce a definition for the normal rank of  two matrices $A$ and $B$.
\begin{mydefinition}[normal rank] For any two matrices $A$ and $B$, the normal rank $\operatorname{nrank}(A, B)$ is defined as
\begin{equation}
\operatorname{nrank}(A, B):=\max _{\beta \in \mathbb{C}} \operatorname{rank}(A-\beta B)
\end{equation}
\end{mydefinition}
\begin{mydefinition}[finite eigenvalues]
A number $\lambda_k\in \mathbb{C}$ satisfying
\begin{equation}
\operatorname{rank}\left(A-\lambda_{k} B\right)<\operatorname{nrank}(A, B)
\end{equation}
is called a finite eigenvalue of the  generalized eigenvalue problem (\ref{AxLBx}).
\end{mydefinition}
Notice that  both $A$ and $B$ in (\ref{AxLBx}) are semi-positive symmetric matrices in our problem, we can consider only the finite eigenvalues in $\mathbb{R}$.
In addition, when $B$ is invertible, we have $\operatorname{nrank}(A, B)=\operatorname{rank}(B)$ and  the above defined finite eigenvalues coincide with the standard eigenvalues of \eqref{invBAx}.

%

We will show that the finite eigenvalues of  (\ref{AxLBx}) are exactly the same as the true \m{eigenvalues of \eqref{eig_problem_discrete_ex}.}
\m{Furthermore, considering the multiplicity of eigenvalues, they are exactly  the eigenvalues of \eqref{eig_problem_discrete}.}

Let us assume that $\{\phi_1|_\Gamma,\phi_2|_\Gamma, \cdots ,\phi_{M-K}|_\Gamma\}$ is a set of linearly independent basis functions of $V_h$,
where $K = \operatorname{dim} W_h - \operatorname{dim} V_h$.
We define the stiffness matrix and mass matrix in $\mathbb{R}^{(M-K)\times(M-K)}$ \m{as}
\begin{equation}
  \tilde A=\big(a(\phi_i,\phi_j)\big),\quad \tilde B=\big( (\phi_i,\phi_j)_{\Gamma}\big), \quad 1\leq i,j\leq M-K.
\end{equation}
Then, the algebraic eigenvalue problem corresponding to \eqref{eig_problem_discrete} is
\begin{equation}\label{true algebraic eigenvalue problem}
\tilde A \mathbf{x} = \lambda \tilde B \mathbf{x}.
\end{equation}
It is easy to see that the eigenvalues of this problem are the true eigenvalues of \eqref{eig_problem_discrete_ex}.
\m{Further, it is a well-posed approximation of the problem \eqref{eig_problem} while considering the multiplicity of eigenvalues as introduced in \cite{glowinski2008computing}.}
The following theorem gives the relation between the regular generalized eigenvalue problem (\ref{true algebraic eigenvalue problem})
and the singular generalized eigenvalue problem (\ref{AxLBx}).
\begin{mytheorem}
  Suppose that $\{\phi_1, \cdots, \phi_{M-K-1}, \phi_{M-K}, \cdots,\phi_{M}\}$ forms a basis of the finite element space $W_h$ and
 $\{\phi_1|_{\Gamma},\cdots,\phi_{M-K}|_{\Gamma}\}$ forms a basis of $V_h$.
Then the generalized eigenvalue problem (\ref{AxLBx}) and (\ref{true algebraic eigenvalue problem}) share the same finite eigenvalues.
\end{mytheorem}

\begin{proof}
By the assumption of the theorem, we easily know that there exist a matrix $C:=\{c_{i,j},i=1,\cdots,K,j=1,\cdots,M-K$\}, 
satisfying
\begin{equation}\label{basis_dependence}
\left(\begin{array}{c}
\phi_{M-K+1}|_{\Gamma} \\
\phi_{M-K+2}|_{\Gamma} \\
\vdots \\
\phi_{M}|_{\Gamma}
\end{array}\right)=\left(\begin{array}{cccc}
c_{1,1} & c_{1,2} & \cdots & c_{1, M-K} \\
c_{2,1} & c_{2,2} & \cdots & c_{2, M-K} \\
\vdots & \vdots & \ddots & \vdots \\
c_{K, 1} & c_{\mathrm{K}, 2} & \cdots & c_{K, \mathrm{~M-K}}
\end{array}\right)\left(\begin{array}{c}
\phi_{1}|_{\Gamma} \\
\phi_{2}|_{\Gamma} \\
\vdots \\
\phi_{M-K}|_{\Gamma}
\end{array}\right).
\end{equation}
We write
\begin{equation*}
A=\left(\begin{array}{cc}
A_{11} & A_{12} \\
A_{21} & A_{22}
\end{array}\right), B=\left(\begin{array}{cc}
B_{11} & B_{12} \\
B_{21} & B_{22}
\end{array}\right)
\end{equation*}
where $A_{11} = \tilde A, B_{11} = \tilde B$. Define  a transformation matrix
\begin{equation*}
P=\left(\begin{array}{cc}
I_{(M-K) \times (M-K)} & 0 \\
-C & I_{K \times K}
\end{array}\right)
\end{equation*}
where $I_{k \times k}$ is the $k\times k$ unit matrix \m{for some positive integer $k$}. From (\ref{basis_dependence}) and the definitions of $A$ and $B$, one derives
\begin{equation*}
\bar A := P A P^{T}=\left(\begin{array}{cc}
A_{11} & 0 \\
0 & 0
\end{array}\right), \bar B: = P B P^{T}=\left(\begin{array}{cc}
B_{11} & 0 \\
0 & 0
\end{array}\right).
\end{equation*}
Since $P$ is invertible,  simple arguments in linear algebra indicate that
\begin{equation*}
\operatorname{rank}(A-\lambda B) = \operatorname{rank}(\bar A-\lambda \bar B)=\operatorname{rank}(A_{11}-\lambda B_{11})=\operatorname{rank}(\tilde A-\lambda \tilde B), \quad \forall \lambda\in \mathbb{R}.
\end{equation*}
The fact $\operatorname{rank}(B) = \operatorname{rank}(\tilde B)$ reveals that
\begin{equation}\label{temp0}
\operatorname{rank}(A-\lambda B)<\operatorname{rank}(B)  \Leftrightarrow \operatorname{rank}(\tilde A-\lambda \tilde B)<\operatorname{rank}(\tilde B).
\end{equation}
By the definition of the stiffness matrix $A$ and the mass matrix $B$ of the trace finite element method, we can easily check that $\operatorname{nrank}(A,B)=\operatorname{rank}(B)$ and also
$\operatorname{nrank}(\tilde{A},\tilde{B})=\operatorname{rank}(\tilde{B})$.
This together with~\eqref{temp0} implies the conclusion of the theorem.
\end{proof}


By the above theorem, to obtain the true \m{eigenvalues} of the discrete problem~\eqref{eig_problem_discrete_ex} \m{with their correct multiplicity}, we can
 compute the finite eigenvalues of the generalized (algebraic) eigenvalue problem~\eqref{AxLBx}.
There exist many algorithms in literature to solve a singular generalized eigenvalue problem 
 \cite{van1979computation,wilkinson1979kronecker,demmel1993generalized,muhivc2009singular}.
Recently, an efficient and robust algorithm has been developed by a rank-completing perturbation technique in \cite{hochstenbach2019solving}.
In this method, one constructs a perturbation of the singular problem which \m{shares} the same finite eigenvalues with the original problem. 
Then, the finite eigenvalues can be selected by using the left and right eigenvectors of the perturbed problem to satisfy certain conditions.
In our experiments we use the algorithm in \cite{hochstenbach2019solving} to solve the singular generalized  eigenvalue problem~\eqref{AxLBx}. More details  are referred to \cite{hochstenbach2019solving}.


\begin{remark}
  For the standard trace FEMs on discrete surfaces, we could use the same idea to construct and solve the corresponding generalized eigenvalue problems. We will present some comparisons in  our numerical experiments in Section 6.
\end{remark}

\section{Error analysis}
In this section, we will present  error analysis for the trace finite element method introduced in the previous sections.
Since there is no geometric error induced by the discretization of the surface, the  analysis is simpler than that for
the standard trace finite element methods \cite{olshanskii2009finite}. In the analysis, we ignore the errors due to numerical quadratures for simplicity.
\subsection{The Laplace-Beltrami equation}
%

 We first introduce some notations. Let $d(x)$ be a signed distance function of $\Gamma$. Then the unit normal vector of $\Gamma$ is given by
\begin{equation}\mathbf{n}({x}):=\nabla d({x}),\quad \forall {x} \in \mathcal{N}_\delta
\end{equation}
The closest point projection operator $\mathbf{p}(x)$ is defined as
\begin{equation}\mathbf{p}({x})={x}-d({x}) \mathbf{n}({x}), \quad \forall  {x} \in \mathcal{N}_\delta.
\end{equation}
For a smooth surface with bounded mean curvature, we can choose $\delta$ small enough to make the projection uniquely defined for all ${x} \in \mathcal{N}_\delta$.
Given a function $v:\Gamma \rightarrow \mathbb{R}$, we  define its natural extension $v^e$ in  $\mathcal{N}_\delta$ as,
\begin{equation}v^e({x})={v}\left(\mathbf{p}({x})\right)={v}({x}-d({x}) \nabla d({x})), \quad \forall {x} \in \mathcal{N}_\delta.
\end{equation}
Suppose that the mesh size $h$ is small enough so that
 $ \omega_h \subset \mathcal{N}_\delta.$
Denote by $\Pi^{}_{h}: C\left({\omega_{h}}\right) \rightarrow W_h$ be the standard Lagrange interpolation operator \m{of degree up to $k$} in
the bulk region. In the following, we use $x\lesssim y$ to \m{represent $x\leq c_0 y$} for some constant $c_0$ independent of the mesh size $h$. In the analysis, we assume $c\geq 0$ in the problem \eqref{m0}.

The following interpolation result is standard in the finite element theory (e.g. \cite{ciarlet2002finite}).
\begin{mylemma}\label{Interpolation_appro}
Let $0\leq m\leq k+1$ and $\hv\in H^{k+1}(\omega_h)$, we have
\begin{equation}
 \|\hv-\Pi^{}_h \hv\|_{H^m(\omega_h)}\lesssim h^{k+1-m}\|\hv\|_{H^{k+1}(\omega_h)}.
\end{equation}
This inequality also holds elementwisely, i.e.  for any $T\in \mathcal{T}_\Gamma$,
\begin{equation}
 \|\hv-\Pi^{}_h\hv\|_{H^m(T)}\lesssim h^{k+1-m}\|\hv\|_{H^{k+1}(T)}.
\end{equation}
\end{mylemma}

The next lemma is a modified version of the trace inequality.
\begin{mylemma}
\label{Modified scaling trace inequality}
Suppose that the surface $\Gamma$ is smooth and $\max \limits_{x\in \Gamma}(|\kappa_1(x)|+|\kappa_2(x)|)<C_{\kappa}$,
\m{where $\kappa_1$ and $\kappa_2$ represent two principal curvatures.}
Let $F_T$ be the intersection between  $\Gamma$ and an element $T$,
then
\begin{equation}\label{trace inequallity}
\|w\|_{0,F_T}^2 \lesssim h^{-1}\|w\|^2_{0,T}+h\|w\|^2_{1,T},
\end{equation}
 holds for all $ w\in H^1(T)$ when $h$ is smaller than a positive number $c_0(C_\kappa)$.
\end{mylemma}
\begin{proof}
  We first map the element $T$ by an affine mapping to a reference element $\tilde{T}$  and denote by $\tilde{F}_T$ the image of $F_T$.
  Assume $\tilde F_T$ divides $\tilde T$ into two subsets $\tilde T_{1}$ and $\tilde T_{2}$ where $T_1$ is shape regular. The coordinates in the reference domain are denoted
 as $(\xi_1,\xi_2,\xi_3)$.   Let $\boldsymbol{n}=(n_{\xi_1},n_{\xi_2},n_{\xi_3})$ be
  the outward unit normal of the boundary of $\tilde T_{1}$. Without loss of generality, we could assume that there is at least one point
  on $\tilde{F}_T$ such that $\boldsymbol{n}=(0,0,1)^T$. This can be done by a simple rotation of $\tilde{T}$ if the condition does not hold.
 Since the size of $\tilde{T}$ is unit, the sum of the principle curvatures of $\tilde{F}_T$ should be smaller than $C_{\kappa}h$. Then
 we have the out normal $|\boldsymbol{n}-(0,0,1)^T|\leq C_{\kappa} h$  on $\tilde{F}_T$. It is easy to see that 
 we have $n_{\xi_3} \geq 1-C_{\kappa}h$.

Then,  by the divergence theorem, we have
  \begin{equation}
  \begin{aligned}
    2 \int_{\tilde{T}_{1}} w \frac{\partial w}{\partial \xi_3} \mathrm{d} V &=\int_{\tilde{T}_{1}} \operatorname{div}\left(0, 0, w^{2}\right)^T \mathrm{d} V=\int_{\partial \tilde{T}_{1}} \boldsymbol{n} \cdot\left(0, 0,w^{2}\right)^T \mathrm{d} S \\
    &\geq (1-C_{\kappa} h)\int_{\tilde{F}_T} w^{2} \mathrm{~d} S+\int_{\partial \tilde{T}_{1} \backslash \tilde{F_T}} n_{\xi_3} w^{2} \mathrm{~d} S.
    \end{aligned}
  \end{equation}
  From Cauchy-Schwarz' inequality and the well-known trace inequality,
  \begin{equation}
    \|w\|_{0, \partial \tilde{T}}^{2} \leq C\|w\|_{0, \tilde{T}}\|w\|_{1, \tilde{T}}, \quad \forall w \in H^{1}(\tilde{T}),
  \end{equation}
  one derives
  \begin{equation}
\begin{aligned}
 (1-C_{\kappa} h)\|w\|_{0, \tilde{F}_T}^{2} & \leq 2\|w\|_{0, \tilde{T}_{1}}\|w\|_{1, \tilde{T}_{1}}+\|w\|_{0, \partial \tilde{T}_{1} \backslash \tilde{F_T}}^{2} \\
& \leq 2\|w\|_{0, \tilde{T}_{1}}\|w\|_{1, \tilde{T_{1}}}+\|w\|_{0, \partial \tilde{T}}^{2} \leq C\|w\|_{0, \tilde{T}}\|w\|_{1, \tilde{T}}\\
& \leq \frac{C}{2}(\|w\|_{0, \tilde{T}}^2+\|w\|_{1, \tilde{T}}^2).
\end{aligned}
\end{equation}
One could easily choose a small $h$  such that  $1>1-C_{\kappa} h\geq C_0>0$. Then
the result of the lemma  follows by a standard scaling argument since $T$ is shape regular.
\end{proof}

The following lemma gives some \m{estimates} on the extensions.
\begin{mylemma} 
\label{Hm_extension}
Under the condition of Lemma~\ref{Modified scaling trace inequality},
suppose $u\in H^{k+1}(\Gamma)$ with $k\geq 1$ 
 and $m=0,1$, then we have
\begin{equation}\label{trace_int_relation}
   \|u^e\|_{H^{m}(\mathcal{N}_\delta)} \lesssim h^{1/2}\|u\|_{H^{m}(\Gamma)},
\end{equation}
and
\begin{equation}\label{trace_int_relation2}
  \|u-\Pi^{}_h u^e\|_{H^{m}(\Gamma)}\lesssim h^{k+1/2-m} \|u^e\|_{H^{k+1}(\omega_h)}.
\end{equation}
\end{mylemma}
\begin{proof}
Resort to (3.17) and (3.18) in \cite{olshanskii2009finite},
\begin{equation}
\left\|u^{e}\right\|_{L^{2}\left(\mathcal{N}_{\delta}\right)} \lesssim \sqrt{h}\|u\|_{L^{2}(\Gamma)}, \qquad
\left\|\nabla u^{e}\right\|_{L^{2}\left(\mathcal{N}_{\delta}\right)} \lesssim \sqrt{h}\|\nabla u\|_{L^{2}(\Gamma)}.
\end{equation}
By the two equations, we easily get
\begin{equation*}
\left\|u^{e}\right\|^2_{L^{2}\left(\mathcal{N}_{\delta}\right)}
 +\left\|\nabla u^{e}\right\|^2_{L^{2}\left(\mathcal{N}_{\delta}\right)}
 \lesssim {h}\|u\|^2_{L^{2}(\Gamma)}
 + {h}\|\nabla_\Gamma u\|^2_{L^{2}(\Gamma)}
 \end{equation*}
This completes the proof of (\ref{trace_int_relation}).
Recall lemma  \ref{Modified scaling trace inequality}, 
\begin{equation}
\|w\|_{L_{2}(F_T)} \lesssim h_{}^{-\frac{1}{2}}\|w\|_{L_{2}( T)}+h_{}^{\frac{1}{2}}\left\|\nabla_{\Gamma} w\right\|_{L_{2}( T)}, \quad \forall w \in H^{1}(T).
\end{equation}
Together with lemma \ref{Interpolation_appro}, we write
\begin{equation*}
\begin{split}
        \left\|u-\Pi^{}_T u^e \right\|_{L^2(F_T)}
\lesssim& h^{-1/2}\left\|u^e-\Pi^{}_T u^e \right\|_{L^2(T)}+h^{1/2}\left\|\nabla (u^e-\Pi^{}_T u^e)\right\|_{L^2(T)}\\
\lesssim &h^{k+1/2} \left\|u^e \right\|_{H^{k+1}(T)}.
\end{split}
\end{equation*}
Similar estimates follow similarly for the semi-$H^1$ norm,
\begin{equation*}
\begin{split}
\left\|\nabla (u-\Pi^{}_T u^e) \right\|_{L^2(F_T)}
\lesssim& h^{-1/2}\left\|\nabla(u^e-\Pi^{}_T u^e) \right\|_{L^2(T)}+h^{1/2}\left\|\nabla^2 (u^e-\Pi^{}_T u^e)\right\|_{L^2(T)}\\
\lesssim &h^{k-1/2} \left\|u^e \right\|_{H^{k+1}(T)}.
\end{split}
\end{equation*}
The assertion (\ref{trace_int_relation2})  follows by summation of the above equations for all $T\in \omega_h$.
\end{proof}
\begin{mylemma}[Approximation in $H^{1}$ norm]\label{appro_H1}
Under the condition of Lemma~\ref{Modified scaling trace inequality},
suppose that ${u} \in$ $H^{k+1}(\Gamma)$ with $k\geq 1$,
then we have
\begin{equation}
  \inf _{v_h \in {V}_h^{}}\left\|{u} - v_h \right\|_{H^{1}(\Gamma)} \lesssim h^{k}\|{u}\|_{H^{k+1}(\Gamma)}. 
\end{equation}
\end{mylemma}
\begin{proof}
Applying Lemma \ref{Hm_extension} directly yields
\begin{equation*}
\begin{split}
 \inf _{v_h \in {V}_h^{}}\left\|{u} - v_h \right\|_{H^{1}(\Gamma)} &\lesssim \|u-\Pi^{}_h u^e\|_{H^1(\Gamma)} 
    \lesssim h^{k-1/2}\|u^e\|_{H^{k+1}(\omega_h)}\\&
    \lesssim h^{k-1/2}\|u^e\|_{H^{k+1}(\mathcal{N}_\delta)}
    \lesssim h^{k}\|u\|_{H^{k+1}(\Gamma)}.
\end{split}
\end{equation*}


\end{proof}
\begin{mytheorem}[A-priori error estimates]\label{error_estimates}
Suppose  the condition of Lemma~\ref{Modified scaling trace inequality} holds and the solution $u$ of the problem~(\ref{m0}) with $c\geq 0$ is in $H^{k+1}(\Gamma)$. Let $u_h \in {V}_{h}^{}$ be the solution of the finite element discretization (\ref{weak_form_apxm}). We assume both $u$ and $u_h$ satisfy the zero-average condition on $\Gamma$ when $c=0$.
   Then we have
$$
\left\|u-u_h\right\|_{L_{2}(\Gamma)}+h\left\|\nabla_{\Gamma}\left(u-u_h\right)\right\|_{L_{2}(\Gamma)} \lesssim h^{k+1}\|{u}\|_{H^{k+1}(\Gamma)}.
$$
\end{mytheorem}
\begin{proof} We only prove the theorem when $c>0$. When $c=0$, the proof is similar with the assumption that $\int_{\Gamma}u ds=\int_{\Gamma}u_h ds=0$.

\m{Define $\hat a(u,v):= a(u,v)+(cu,v)_\Gamma$.}
Take $v=v_h$ in (\ref{weak_form}) and then subtract (\ref{weak_form_apxm}), arrive at
\begin{equation*}
  \hat a(u-u_h,v_h) = 0, \quad \forall v_h \in V_h^{}(\Gamma).
\end{equation*}
According to the continuity and ellipticity of bilinear operator $\hat{a}$,
\begin{equation*}
\begin{split}
  \left\|u-u_h \right\|^2_{H^1(\Gamma)} &\lesssim \hat a(u-u_h,u-u_h) \\
    &=\hat a(u-u_h,u-v_h)+\hat a(u-u_h,v_h-u_h)\\
    &\lesssim  \left\|u-u_h \right\|_{H^1(\Gamma)}\left\|u-v_h \right\|_{H^1(\Gamma)}
\end{split}
\end{equation*}
then we have
\begin{equation}\label{H1_error_estimate}
\left\|\nabla_{\Gamma}\left(u-u_h\right)\right\|_{L_{2}(\Gamma)} \leq \left\|u-u_h \right\|_{H^1} \lesssim \left\|u-v_h \right\|_{H^1}\\
     \lesssim h^{k}\|{u}\|_{H^{k+1}(\Gamma)}
\end{equation}
Next, we use the Aubin-Nitsche duality technique to estimate the $L_2$ error.
We now consider an auxiliary problem,
\begin{equation}\label{dual_pro}{z} \in H_{}^{1}(\Gamma): \quad \hat a(z,w)=\int_{\Gamma}\left({u}-u_h\right) w \quad \forall w \in H_{}^{1}(\Gamma)\end{equation}
and its finite element approximation problem,
\begin{equation}
{z_h} \in V_{h}^{}(\Gamma): \quad \hat a(z_h,w_h)=\int_{\Gamma}\left({u}-u_h\right) w_h \quad \forall w_h \in V_{h}^{}.
\end{equation}
Take $w=u-u_h$ in (\ref{dual_pro}),
\begin{equation*}
\begin{split}
  \left\|u-u_h \right\|^2_{L^2} &= \hat a(z,u-u_h) 
    =\hat a(z-z_h,u-u_h)\\
    &\lesssim  \left\|z-z_h \right\|_{H^1(\Gamma)}\left\|u-u_h \right\|_{H^1(\Gamma)}\\
    &\lesssim  h^{k+1}\left\|u-u_h \right\|_{L^2}\left\|u \right\|_{H^{k+1}(\Gamma)},\\
\end{split}
\end{equation*}
where in the last inequality we have used \m{$H^1$ error estimate \eqref{H1_error_estimate}} and the regularity property   $\|z\|_{H^2(\Gamma)}\leq \left\|u-u_h \right\|_{L^2}$ for the problem~\eqref{dual_pro}.
This completes the proof.
\end{proof}

\subsection{The Laplace-Beltrami eigenvalue problem}
For the error estimate for the Laplace-Beltrami eigenvalue problem, we adopt the standard approach  using the spectral approximation theory
for compact operators \cite{babuvska1991eigenvalue,boffi2010finite,sun2016finite}.
This is based on the error analysis for the Laplace-Beltrami equation in the previous subsection.

We first introduce some operators as in \cite{babuvska1991eigenvalue,sun2016finite} for the weak formula of the Laplace-Beltrami equations.
In this subsection, we consider only the nontrivial eigenvalues. We denote the function space
$$H^1_c(\Gamma)=\{v\in H^1(\Gamma)\ |\ \int_{\Gamma} v ds=0.\}$$
and
$$L_c^2(\Gamma)=\{v\in L^2(\Gamma)\ |\ \int_{\Gamma} v ds=0.\}
$$
Then an operator $T:L_c^2(\Gamma)\mapsto L_c^2(\Gamma)$ is defined as follows.  For any $f\in L_c^2(\Gamma)$,
$Tf\in H^1_c(\Gamma)\subset L_c^2(\Gamma) $ is a function satisfies
\begin{equation}
a(T f, v) =(f,v)_{\Gamma}, \qquad \forall v\in H^1_c(\Gamma).
\end{equation}
This is the Laplace-Beltrami equation~\eqref{weak_form} (with $c=0$). It is easy to know that the problem is well-defined and $\|Tf\|_{H^1}\leq \|f\|_{L^2}$.
Then the Sobolev embedding theorem implies that $T:L_c^2(\Gamma)\mapsto L_c^2(\Gamma)$ is compact. It is easy to check that
the operator $T$ is also self-adjoint:
\begin{align*}
(T f, g)_{\Gamma}=(g,Tf)_{\Gamma}=a(Tg, Tf)=a(Tf,Tg)=(f,Tg)_{\Gamma}.
\end{align*}
Furthermore, if $\lambda$ is a nonzero eigenvalue of the Laplace-Beltrami eigenvalue problem~\eqref{eig_problem_weak}, then
$\lambda^{-1}$ is the eigenvalue of $T$, corresponding to the same eigenfunctions.

 Denote by $V_{h,c}=V_h\cap H^1_c$, i.e.
$V_{h,c}=\{v_h\in V_h : \int_{\Gamma} v ds=0\}$. We can define a discrete operator $T_h: L_c^2(\Gamma)\mapsto L_c^2(\Gamma)$ analogously.
For any $f\in L^2_c(\Gamma)$, $T_h f\in V_{h,c}\subset L_c^2(\Gamma)$ is determined by
\begin{equation}
a(T_h f, v_h) =(f,v_h)_{\Gamma}, \qquad \forall v_h \in V_{h,c}.
\end{equation}
 It is easy to see that $T_h f$ is the finite element solution of \eqref{weak_form_apxm} with $c=0$.
$T_h$ is a self-adjoint operator. If $\lambda_h$ is a nonzero eigenvalue of the problem~\eqref{eig_problem_discrete}, then
$\lambda^{-1}_h$ is the eigenvalue of $T_h$.
In addition, the error analysis
in last subsection implies that
$$\|Tf-T_h f\|_{H^1}\leq C h^k \|f\|_{k-1},$$
where we use the well-known regularity result that $\|u\|_{H^{k+1}}\lesssim \|f\|_{k-1}$ for the Laplace-Beltrami equation.

We have the following optimal error estimate for the trace \m{finite} element method for the eigenvalue problem.
\begin{mytheorem}\label{error_eig}
Let $\lambda$ be a nonzero eigenvalue of the problem~\eqref{eig_problem_weak} with multiplicity $m$. Let $u_i$, $i=1,\cdots, m,$
be orthogonal eigenfunctions corresponding to $\lambda$. We suppose that $u_i\in H^{k+1}(\Gamma)$.
Let $\lambda_{h,i},i=1,\cdots,m,$ be the  solutions of the discrete problem (\ref{eig_problem_discrete}) approximating to $\lambda$.
Then we have the following error estimates
$$
\left|\lambda-\hat{\lambda}_{h}\right| \lesssim h^{2 k} ,
$$
where $\hat{\lambda}_{h}=\frac{1}{m}\sum_{i=1}^m\lambda_{h,i}$.
%
\end{mytheorem}
\begin{proof}
By the definition of the operators $T$ and $T_h$, we can apply the Babuska-Osborn theory. Denote $E=\text{span}\{u_{1},\cdots, u_{m}\}$.
 By  Corollary 9.6. in \cite{boffi2010finite}, noticing that both $T$ and $T_h$ are self-adjoint,  we obtain
\begin{align*}
|\lambda-\hat{\lambda}_h|&\leq C \Big(\sup_{u\in E, \|u\|_{L^2(\Gamma)}=1} \inf_{v_h\in V_{h,c}}\|u-v_h\|_{H^1}\Big)^2 \\
 &\leq
C \Big(\sup_{u\in E, \|u\|_{L^2(\Gamma)}=1} h^k \|u\|_{H^{k+1}}\Big)^2 \lesssim
h^{2k}.
\end{align*}
\end{proof}
\begin{remark}
By definition, the error estimate holds for the nonzero true eigenvalues of the embedded problems \eqref{eig_problem_ex} and
\eqref{eig_problem_discrete_ex}.
In addition,  error analysis for eigenfunctions can be done in a standard way \cite{boffi2010finite},
which is omitted here for simplicity.
\end{remark}

\section{Numerical experiments}
In this section, we give some numerical examples to show the convergence behaviour of our method and compare with the original version of the trace finite element method. We implement the method in the finite element package DROPS \cite{DROPS} and use the package~PHG\cite{PHG} to do numerical integration on surfaces.

We first test our method by solving the Laplace-Beltrami equation~\eqref{m0}.

{\it Example 1.}
We solve the equation (\ref{m0}) on a unit sphere $\Gamma$:
$$\Gamma=\left\{x \in \mathbb{R}^{2}\ |\ \| x \|_2=1\right\}$$
which is contained in $\Omega :=[-2,2]^3$ and implicitly presented by the zero level of a function $\operatorname \varphi(x) = \|x\|_2-1$. We set $c=1$ and the right hand side term is given by
$
f(x)=\left(3 x_{1}^{2} x_{2}-x_{2}^{3}\right).
$
The solution of the equation is explicitly given by
\begin{equation*}
u(x)= \frac{|x|^{2}}{12+|x|^{2}}\left(3 x_{1}^{2} x_{2}-x_{2}^{3}\right), \quad x \in \Gamma.
\end{equation*}

For the triangulation we partition $\Omega$ uniformly into $N^3$ cubes and then each of them is subdivided into six tetrahedra.
We solve the finite element problem~\eqref{weak_form_apxm_ex} to get the numerical solution. In comparison, we also solve the problem~\eqref{m0}
by using the isoparametric trace finite element method \cite{grande2016higher}.
We compute the numerical errors in $L^2$ and $H^1$ norms. The  experimental orders of convergence (EOC) are computed accordingly.

A numerical solution by our method is shown in Figure~\ref{LB_equ_accuracycheck} ($k=2$, $N=64$). The numerical errors are given in Table~\ref{e1P1} and \ref{e1P2}.
 It is found that the geometrically consistent method (exTraceFEM) has optimal  convergence rate, the same as the isoparametric method (isoTraceFEM). We can also see that the errors by the exTraceFEM is almost the same as
that by the isoTraceFEM on a triangulation with the same mesh size when $k=1$(linear finite element case). This implies that the geometric error is neglectable for the low order method. However,  for a higher order case ($k=2$), the numerical errors of exTraceFEM is much smaller than those of the isoparametric finite element method. This indicates that the effect of the geometric consistency is significant for
high order finite element methods.
\begin{figure}[!htp]
    \centering
    \includegraphics[width=7cm]{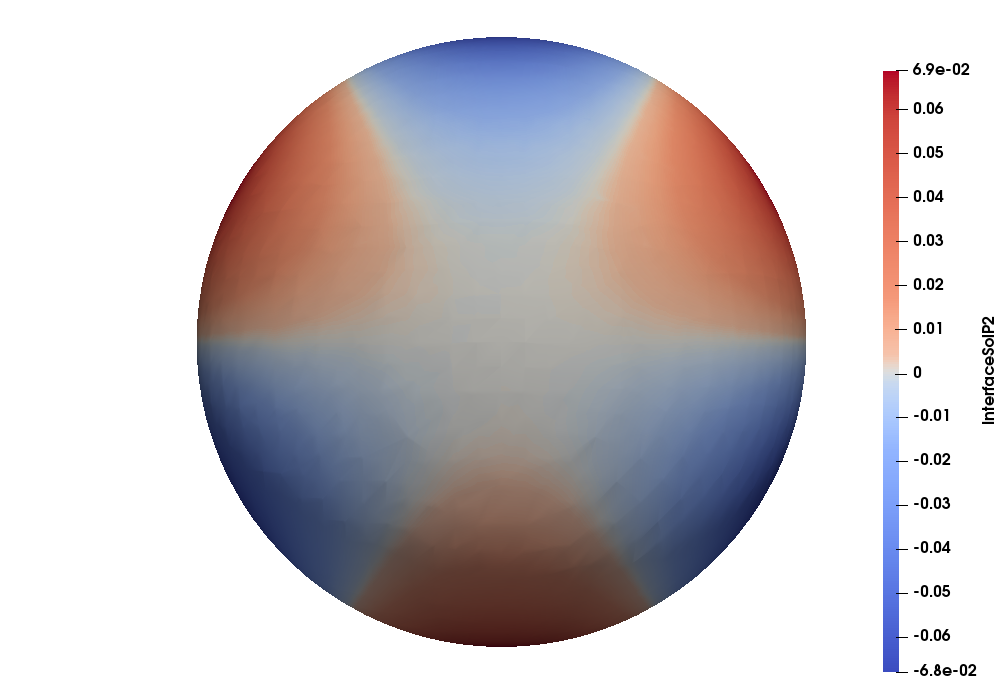}
    \caption{Numerical solution of the LB equation on a spherical surface}\label{LB_equ_accuracycheck}
\end{figure}

\begin{table}[H]
	\centering
    \caption{Comparison of exTraceFEM and isoparametric TraceFEM (Example 1, $k=1$).}\label{e1P1}
	\begin{tabular}{c|cc|cc|cc|cc}
		\hline
        \hline
		&\multicolumn{4}{|c|}{isoTraceFEM} & \multicolumn{4}{|c}{exTraceFEM} \\
		\hline
		$N$ & $E(H_1)$ & EOC & $E(L_2)$ & EOC  & $E(H_1)$ & EOC& $E(L_2)$ & EOC\\
		\hline
        4     & 8.73E-1       & - &4.88E-2 & -           & 9.23E-1 & -&6.68E-2 & -\\
        8 & 4.62E-1       & 0.92  &1.99E-2 & 1.29    & 4.76E-1 & 0.96 &2.05E-2 & 1.70\\
        16  & 2.16E-1       & 1.10&5.32E-3 & 1.90     & 2.20E-1 & 1.11&5.12E-3 & 2.00\\
        32 & 1.04E-1       & 1.05 &1.38E-3 & 1.95    & 1.04E-1 & 1.08&1.30E-3 & 1.98 \\
        64  & 5.18E-2       & 1.00 &3.33E-4 & 2.05   & 5.22E-2 & 0.99&3.13E-4 & 2.05 \\
        \hline
	\end{tabular}
\end{table}

\begin{table}[H]
	\centering
    \caption{Comparison of exTraceFEM and isoparametric TraceFEM (Example. 1, $k=2$).}\label{e1P2}
	\begin{tabular}{c|cc|cc|cc|cc}
		\hline
        \hline
		&\multicolumn{4}{|c|}{isoTraceFEM} & \multicolumn{4}{|c}{exTraceFEM} \\
		\hline
		$N$ & $E(H_1)$ & EOC &$E(L_2)$ & EOC  & $E(H_1)$ & EOC& $E(L_2)$ & EOC \\
		\hline
        4    & 2.52E-1       & -    &4.57E-2 & -  & 1.70E-1 & - &      1.91E-1 & - \\
        8   & 4.26E-2       & 2.56 &3.05E-3 & 3.91   & 1.08E-2 & 3.98&5.46E-4 & 8.45 \\
        16  & 9.94E-3       & 2.10 &3.48E-4 & 3.13   & 2.56E-3 & 2.08&6.49E-5 & 3.07 \\
        32  & 2.61E-3       & 1.92 &4.76E-5 & 2.87   & 6.61E-4 & 1.95&8.71E-6 & 2.90 \\
        64  & 6.55E-4       & 1.99  &5.98E-6 & 2.99   & 1.64E-4 & 2.01&1.07E-6 & 3.03\\
        \hline
	\end{tabular}
\end{table}

We then present some numerical examples for the \m{Laplace-Beltrami} eigenvalue problem.

{\it Example 2.} In this example, we consider the eigenvalue problem~\eqref{eig_problem} on the unit spherical surface.
It is well known that the eigenvalues of the Laplace-Beltrami operator on the surface is  given by
$$\lambda^m = \{m(m-1), m\geq 1\}$$
 with a multiplicity of $p_m = 2m-1$. 
  We solve the problem (\ref{eig_problem_discrete_ex}) numerically.
 Meanwhile, we also solve the Laplace-Beltrami eigenvalue problem using a standard trace finite element method in \cite{olshanskii2009finite}.

\m{We first compute the generalized eigenvalues of $A$ and $B$ using the Cholesky factorization method.}
  The numerical results
 are shown in Table~\ref{ISO_k2N32old} and Table~\ref{NISO_k2N32old}, which
 are respectively for the standard trace finite element method and the geometrically consistent method. For simplicity, we only show the
 first six eigenvalues. Here we choose $k=2$ and $N=32$. There are 2604 freedoms totally in  the finite element space $W_h$. Our numerical results show that there are 448 false eigenvalues for the standard trace finite element method.
 This is actually \m{because} the dimension of $W_h$ is much larger than \m{that of} the space $V_h$ on a discretized surface $\Gamma_h$, as discussed in Section 3.
 Moreover,  the geometrically consistent trace finite element method  behaves much better than the standard trace finite element method in this case.
 There seems only one false eigenvalue $0.0069093$ which is  close to the trivial eigenvalue $0$. This is consistent with our theoretical result in Lemma~\ref{con_wellpose}. 
 \m{In the same bulk finite element space $W_h$, $\operatorname{dim}$$(V_h)$ of our method is much larger than that of the standard traceFEM, which indicates that the geometrically consistent method can generate much less false eigenvalues.}

We then use the rank-completing perturbation algorithm \cite{hochstenbach2019solving} to compute the finite eigenvalues of the generalized eigenvalue problem~\eqref{AxLBx}.
The numerical results by the standard method(traceFEM) and our method (exTraceFEM) ($k=2$ and $N=32$) are shown in Table~\ref{ISO_k2N32} and Table~\ref{NISO_k2N32}, respectively.
We can see that the false eigenvalues have been selected out as discussed in Section 4. Both the eigenvalues and their multiplicity are computed correctly.
Similar to that for Laplace-Beltrami equation, the accuracy of the geometrically consistent method  is much better than the standard method.
\begin{table}[]\scriptsize
\caption{traceFEM method when $k=2,N=32$}\label{ISO_k2N32old}
\begin{tabular}{{p{0.12\columnwidth}|p{0.8\columnwidth}}}
\hline\hline\vspace{0.01cm}
{\small Eigenvalue} & \vspace{0.01cm}{\small Numerical solutions}
\\ \hline
\vspace{0.01cm}0               &-1.8899201033073 -1.65160247438363 -1.60753159863045 $\cdots$
1.146997797544159 1.173366051077160
1.404691246035232
 \\ \hline
2               &  2.000002341879660 \ 2.000002341895009 \ 2.000010968828590 \  2.008644495856585                                                                                                                                                                                                                                                                                                                                                                         \\ \hline
\vspace{0.01cm}6               & 6.000032528374658  6.000032528374662  6.000062692154771  6.000091817982969 6.000091817982981                                                                                                                                                                                                                                                                                                                                           \\ \hline
\vspace{0.01cm}12              & 12.000169159806013  12.000181120960798  12.000402418432977 12.000402418432987 12.000526221431546 12.000526221431546 12.000714928288357                                                                                                                                                                                                                                                                                                \\ \hline
\vspace{0.03cm}20              & 20.000639807964866 20.000639807964877 20.001210428217536 20.001433468889189 20.002056244148882 20.002056244148889 20.002704995918332 20.002777886515734 20.002777886515734                                                                                                                                                                                                                                                          \\ \hline
\vspace{0.03cm} 30              & 30.001877620894145 30.001877620894149 30.003586778395437 30.003586778395444 30.005578378716464 30.006082304547288 30.007563415744823 30.007563415744837 30.008657097319883 30.008657097319894 30.009729163449663                                                                                                                                                                                                                    \\ \hline
\end{tabular}
\end{table}

\begin{table}[]\scriptsize
\caption{exTraceFEM method when $k=2,N=32$}\label{NISO_k2N32old}
\begin{tabular}{{p{0.12\columnwidth}|p{0.8\columnwidth}}}
\hline\hline\vspace{0.01cm}
{\small Eigenvalue} & \vspace{0.01cm}{\small Numerical solutions}
\\ \hline
0               &-4.5702e-11        0.0069093                                                                                                                                                                                                                                                                                                                                                                                                            \\ \hline
2               & 1.999999999970570 2.000000000028308 2.000000000051764                                                                                                                                                                                                                                                                                                                                                                               \\ \hline
\vspace{0.01cm}6               & 5.999999999933209 5.999999999940492 5.999999999994404 6.000000000001932 6.000000000036823                                                                                                                                                                                                                                                                                                                                           \\ \hline
\vspace{0.01cm}12              & 12.000011666430954 12.000020383250462 12.000024586245313 12.000024898829679 12.000024898928142 12.000034279424140 12.000034279500648                                                                                                                                                                                                                                                                                                \\ \hline
\vspace{0.03cm}20              & 20.000121867977601 20.000121867991975 20.000131606614801 20.000188235821870 20.000188235928746 20.000189756600580 20.000204020147969 20.000235383589722 20.000235383594994                                                                                                                                                                                                                                                          \\ \hline
\vspace{0.03cm}30              & 30.000514523042980 30.000534231291937 30.000534231297287 30.000635177166853 30.000635177208498 30.000771041561322 30.000771041608086 30.000836697770890 30.000836697819057 30.000856733024339 30.001005770449328                                                                                                                                                                                                                    \\ \hline
\end{tabular}
\end{table}

\begin{table}[]\scriptsize
\caption{traceFEM method when $k=2,N=32$}\label{ISO_k2N32}
\begin{tabular}{{p{0.12\columnwidth}|p{0.8\columnwidth}}}
\hline\hline\vspace{0.01cm}
{\small Eigenvalue} & \vspace{0.01cm}{\small Numerical solutions}
\\ \hline
0               &-6.44833099219222e-13                                                                                                                                                                                                                                                                                                                                                                                               \\ \hline
2               &2.000002341879660 2.000002341895009 2.000010968828590                                                                                                                                                                                                                                                                                                                                                                      \\ \hline
\vspace{0.01cm}6               & 6.000032528374658 6.000032528374662 6.000062692154771 6.000091817982969 6.000091817982981                                                                                                                                                                                                                                                                                                                                           \\ \hline
\vspace{0.01cm}12              & 12.000169159806013 12.000181120960798 12.000402418432977 12.000402418432987 12.000526221431546 12.000526221431546 12.000714928288357                                                                                                                                                                                                                                                                                                \\ \hline
\vspace{0.03cm}20              & 20.000639807964866 20.000639807964877 20.001210428217536 20.001433468889189 20.002056244148882 20.002056244148889 20.002704995918332 20.002777886515734 20.002777886515734                                                                                                                                                                                                                                                          \\ \hline
\vspace{0.03cm}30              & 30.001877620894145 30.001877620894149 30.003586778395437 30.003586778395444 30.005578378716464 30.006082304547288 30.007563415744823 30.007563415744837 30.008657097319883 30.008657097319894 30.009729163449663                                                                                                                                                                                                                    \\ \hline
\end{tabular}
\end{table}

\begin{table}[]\scriptsize
\caption{exTraceFEM method when $k=2,N=32$}\label{NISO_k2N32}
\begin{tabular}{{p{0.12\columnwidth}|p{0.8\columnwidth}}}
\hline\hline\vspace{0.01cm}
{\small Eigenvalue} & \vspace{0.01cm}{\small Numerical solutions}
\\ \hline
0               &-4.5702e-11                                                                                                                                                                                                                                                                                                                                                                                                         \\ \hline
2               & 1.999999999970570 2.000000000028308 2.000000000051764                                                                                                                                                                                                                                                                                                                                                                               \\ \hline
\vspace{0.01cm}6               & 5.999999999933209 5.999999999940492 5.999999999994404 6.000000000001932 6.000000000036823                                                                                                                                                                                                                                                                                                                                           \\ \hline
\vspace{0.01cm}12              & 12.000011666430954 12.000020383250462 12.000024586245313 12.000024898829679 12.000024898928142 12.000034279424140 12.000034279500648                                                                                                                                                                                                                                                                                                \\ \hline
\vspace{0.03cm}20              & 20.000121867977601 20.000121867991975 20.000131606614801 20.000188235821870 20.000188235928746 20.000189756600580 20.000204020147969 20.000235383589722 20.000235383594994                                                                                                                                                                                                                                                          \\ \hline
\vspace{0.03cm}30              & 30.000514523042980 30.000534231291937 30.000534231297287 30.000635177166853 30.000635177208498 30.000771041561322 30.000771041608086 30.000836697770890 30.000836697819057 30.000856733024339 30.001005770449328                                                                                                                                                                                                                    \\ \hline
\end{tabular}
\end{table}

We also compute the convergence rate of the eigenvalues. Let $\lambda^{m}(m=1,2,\cdots)$ be \m{the m-th eigenvalue} of the Laplace-Beltrami eigenvalue problem (\ref{eig_problem}),
and $p_m$ is the multiplicity of $\lambda^m$.
We use $\{\lambda_{h,1}^m,\lambda_{h,2}^m,\cdots,\lambda_{h,p_m}^m\}$ to represent the eigenvalues \m{approximating $\lambda^m$} by the discrete eigenvalue problem (\ref{eig_problem_discrete}).
The numerical error is calculated as,
\begin{equation}\label{err_def}
  \operatorname{Error(\lambda^m)} = \frac{1}{p_m}{\sum_{i=1}^{p_m} |\lambda^m-\lambda^m_{h,i}|}.
\end{equation}
We compute the numerical errors for both the standard method and our method.
The results for the cases $k=1$ and $k=2$ are shown in Figure \ref{k1_error} and Figure \ref{k2_error}, respectively. We show the convergence for the first ten eigenvalues for simplicity.
In general the experimental order of convergence is  $2$ in the linear finite element case($k=1$) and
 the order is $4$  when $k=2$, except for the first a few eigenvalues, where the error is close to the machine accuracy. The optimal convergence orders agree with the  error estimates in the previous section.
In addition, some typical eigenfunctions are shown in Figure~\ref{eig1to6}.

\begin{figure}[!h]
  \centering
  \includegraphics[width=6.1cm]{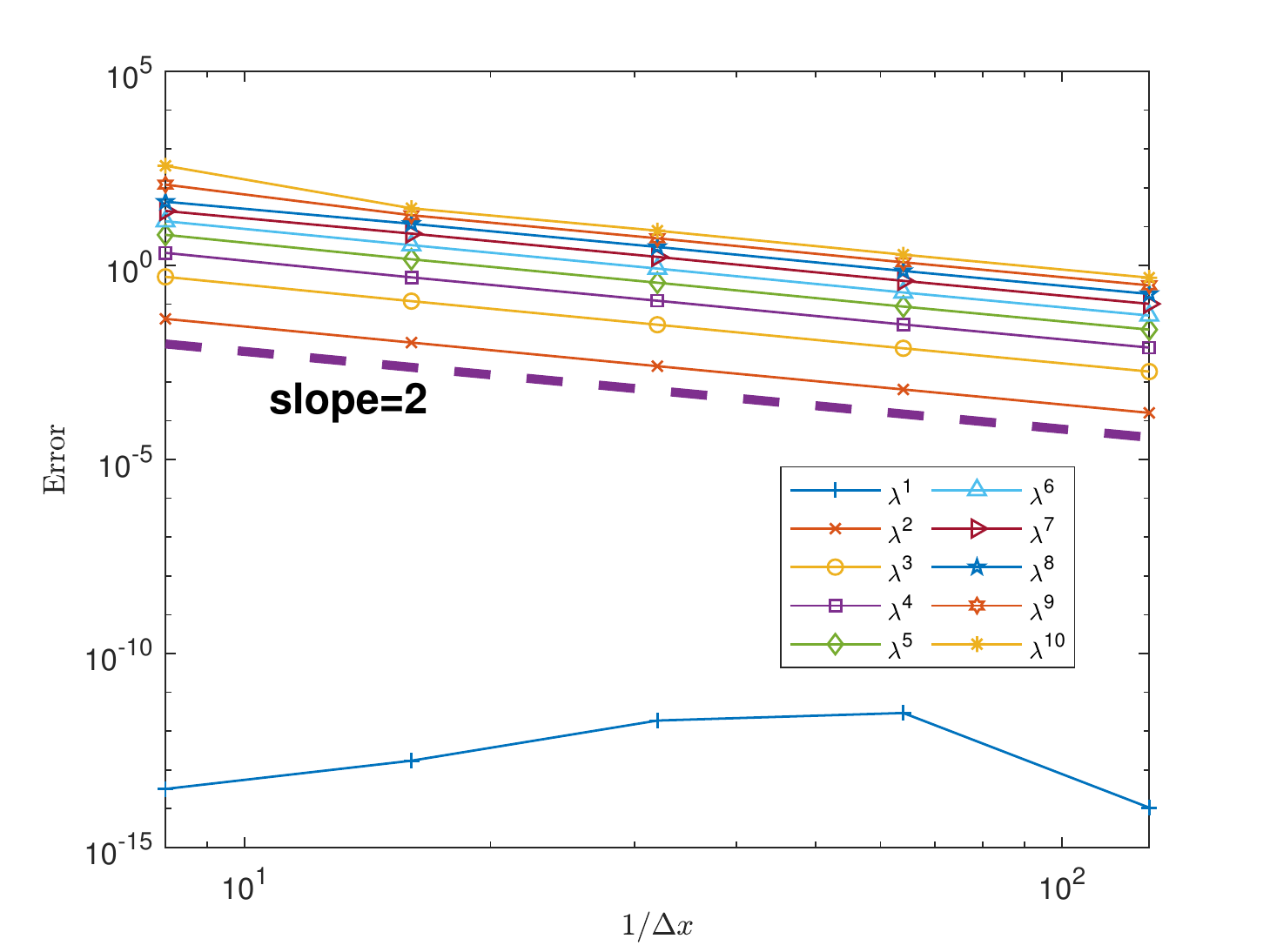}\qquad
  \includegraphics[width=6.1cm]{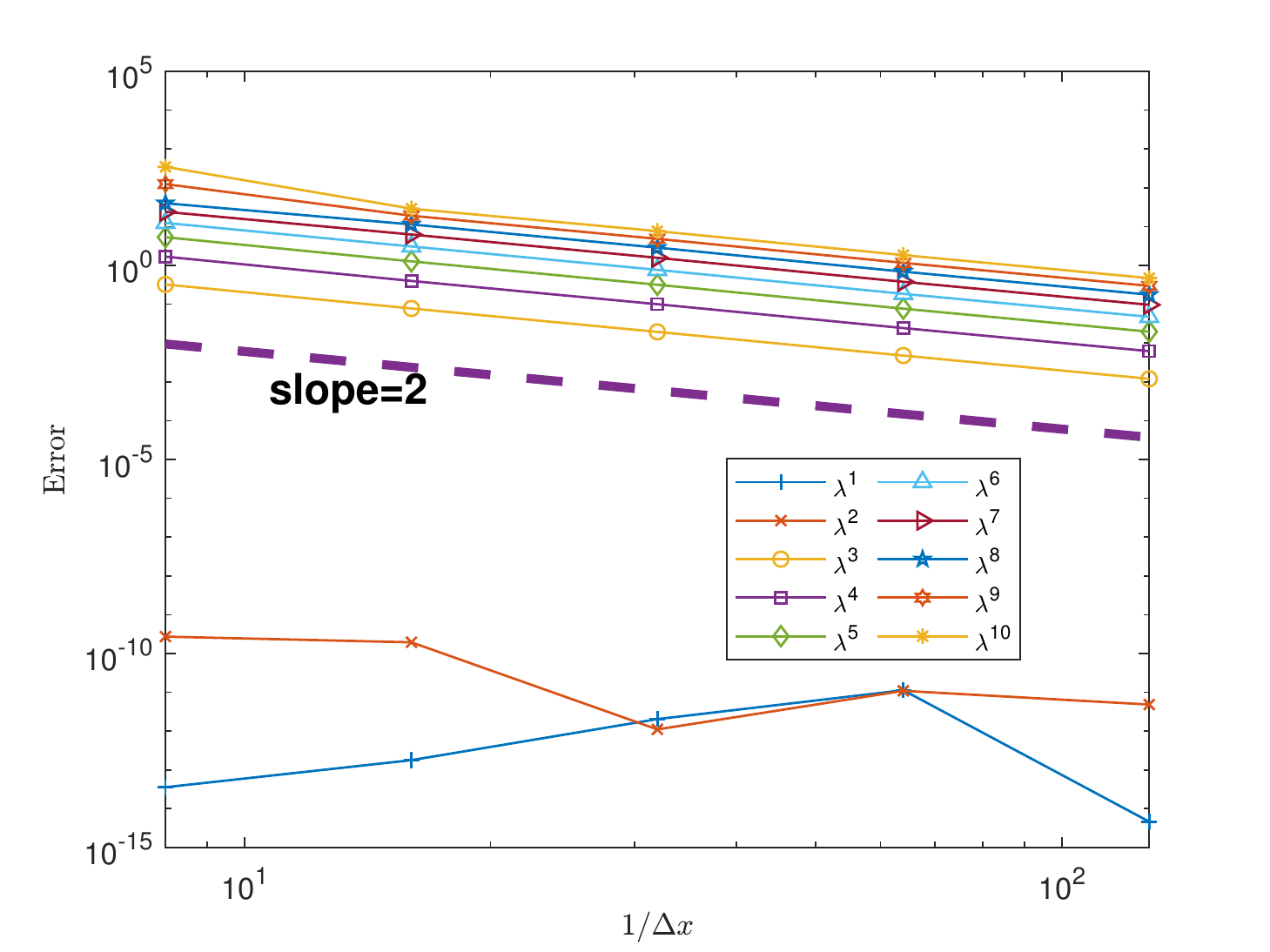}
  \caption{The errors of traceFEM (left) and exTraceFEM (right) for the Laplace-Beltrami eigenvalue problem(Example 2, k=1).}\label{k1_error}
\end{figure}
\begin{figure}[!h]
  \centering
  \includegraphics[width=6.1cm]{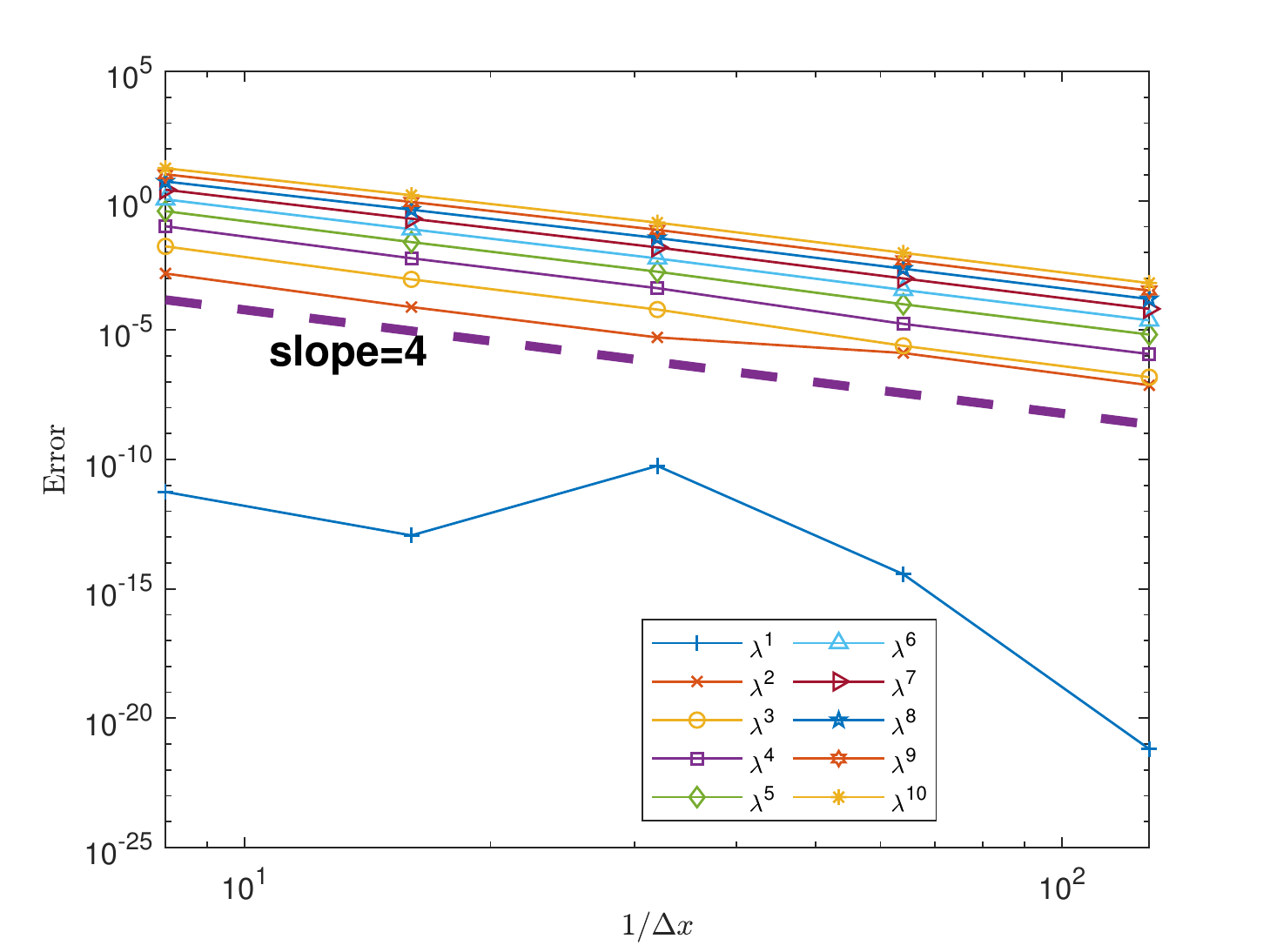}\qquad
  \includegraphics[width=6.1cm]{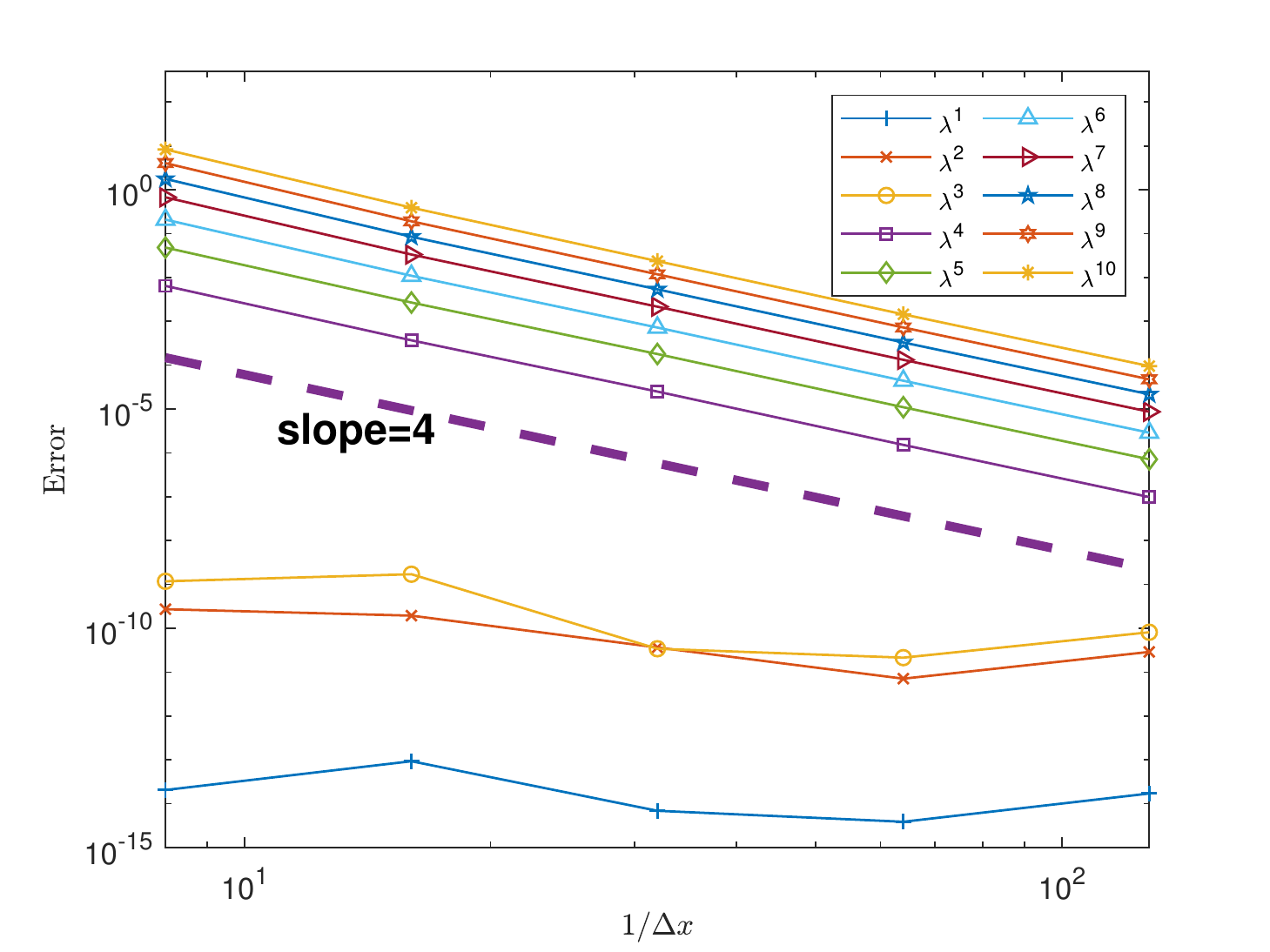}
  \caption{The errors of traceFEM (left) and exTraceFEM (right) for the Laplace-Beltrami eigenvalue problem(Example 2, k=1).}\label{k2_error}
\end{figure}

%
%

\begin{figure}
\centering
\subfigure[$\lambda_{h,1}\approx 0  $]{
\includegraphics[width=0.23\linewidth]{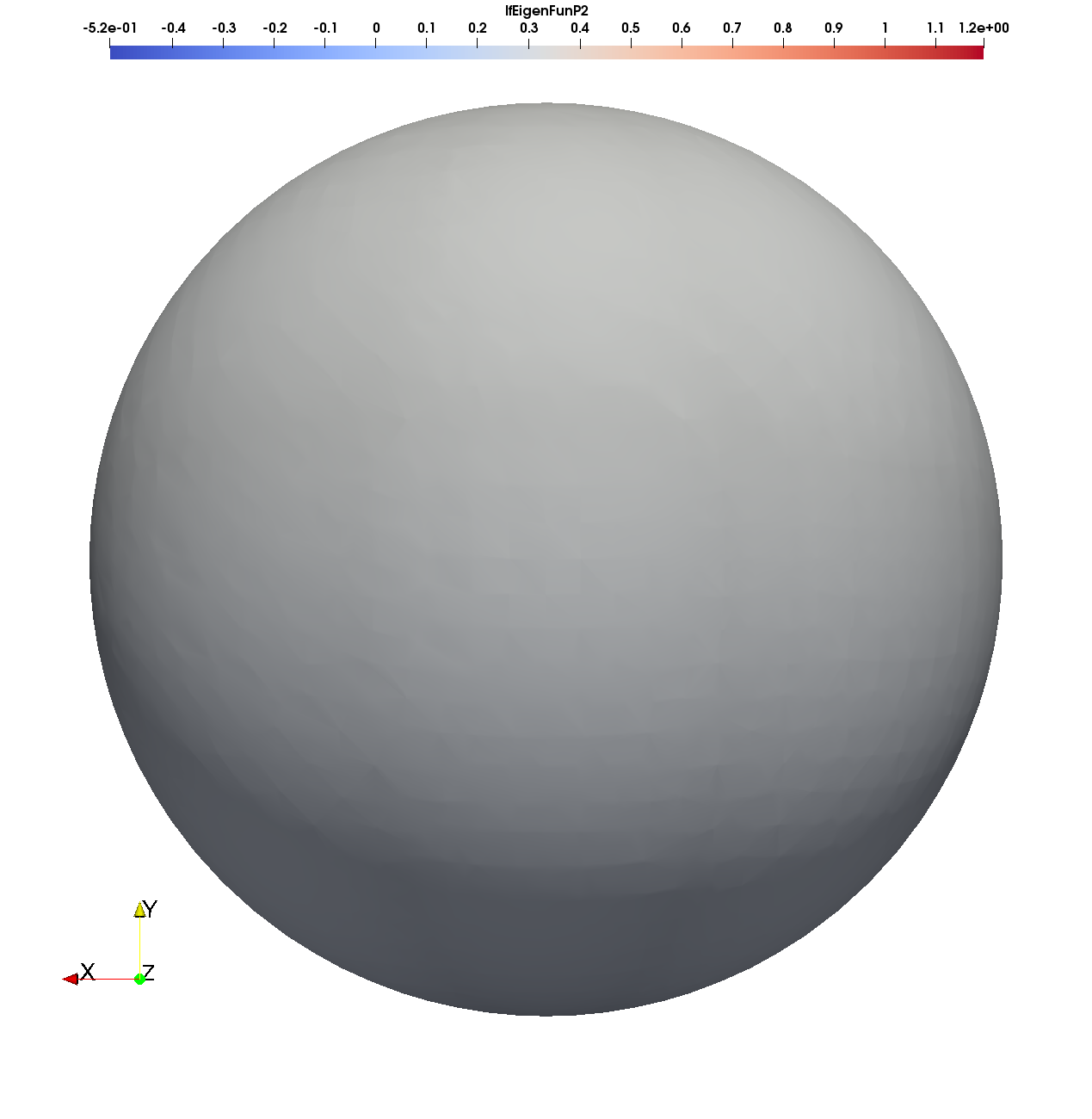}}
\hspace{0.01\linewidth}
\subfigure[$\lambda_{h,2} \approx 2 $]{
\includegraphics[width=0.23\linewidth]{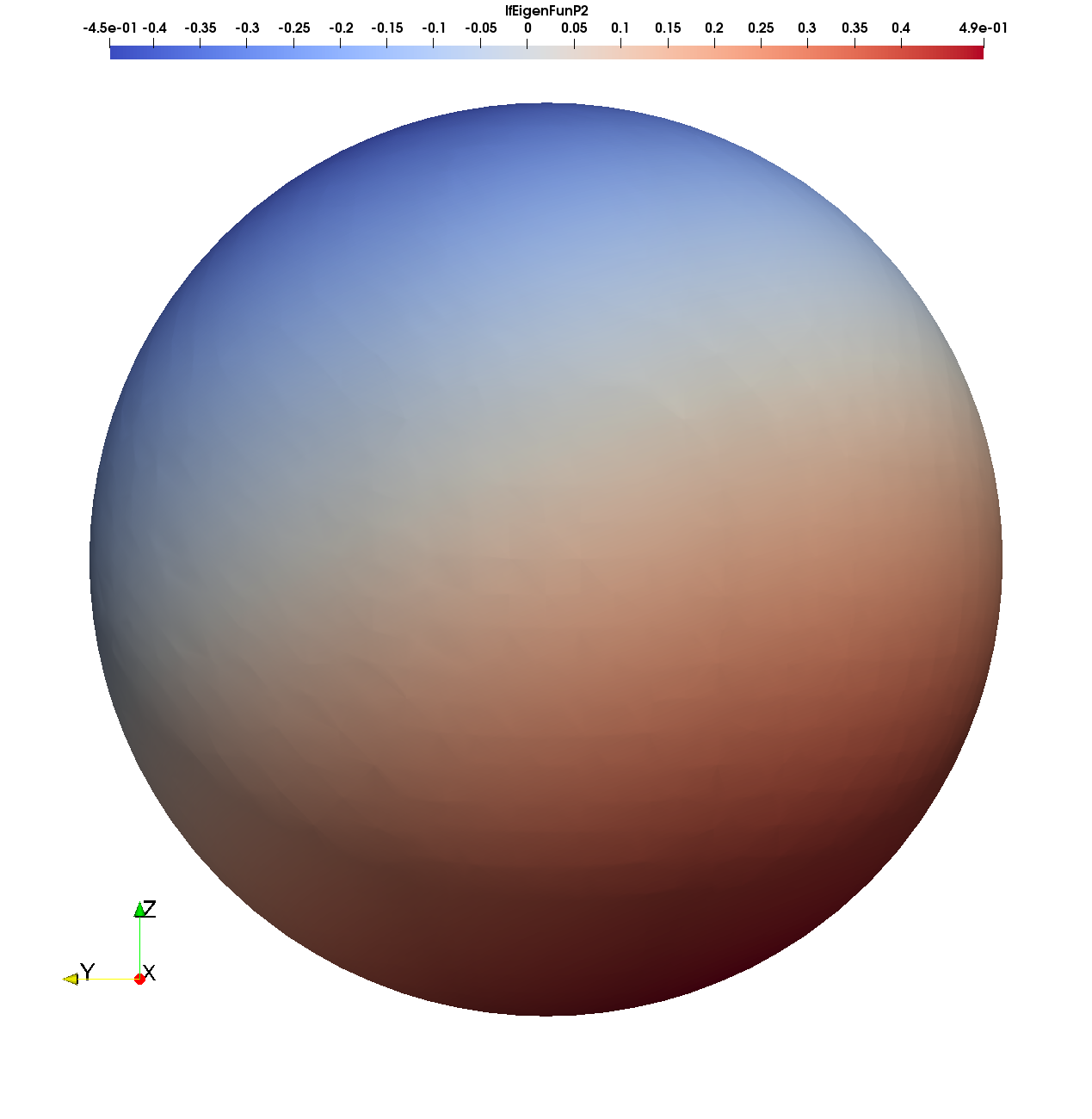}}
\hspace{0.01\linewidth}
\subfigure[$\lambda_{h,3}  \approx 2 
$]{
\includegraphics[width=0.23\linewidth]{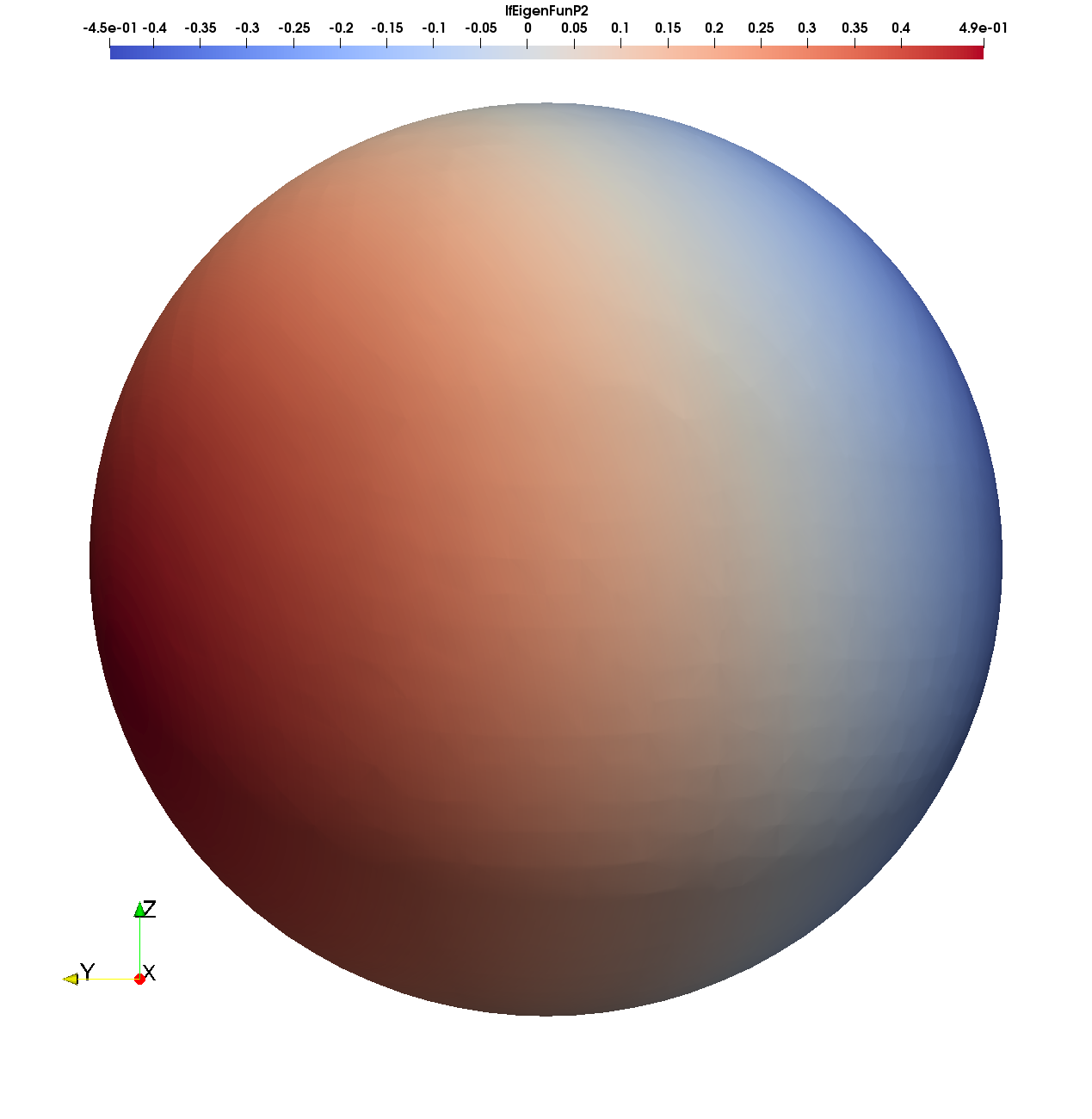}}
\subfigure[$\lambda_{h,4}  \approx 2
$]{
\includegraphics[width=0.23\linewidth]{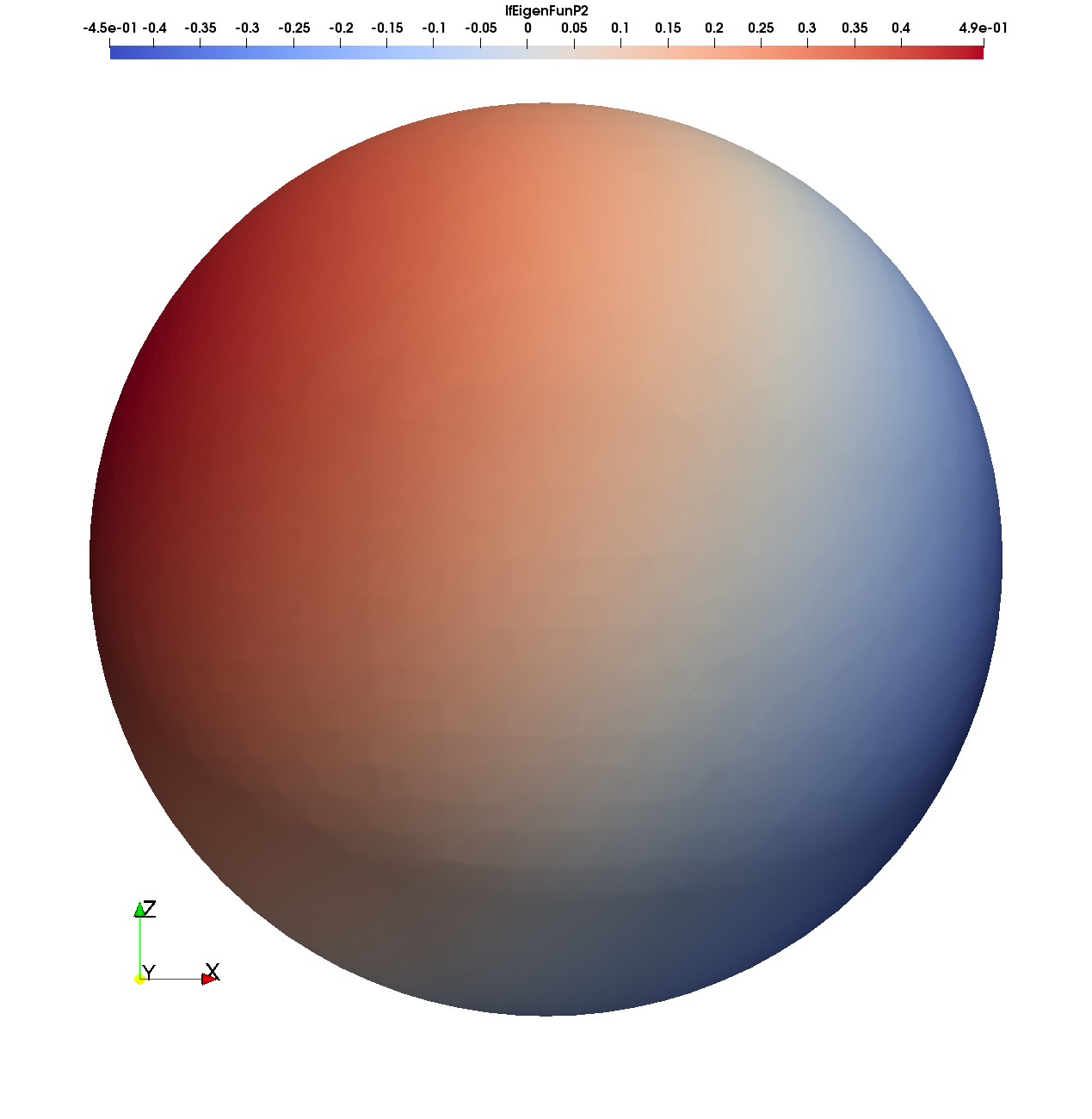}}
\hspace{0.01\linewidth}
\subfigure[$\lambda_{h,5} \approx 6 
$]{
\includegraphics[width=0.23\linewidth]{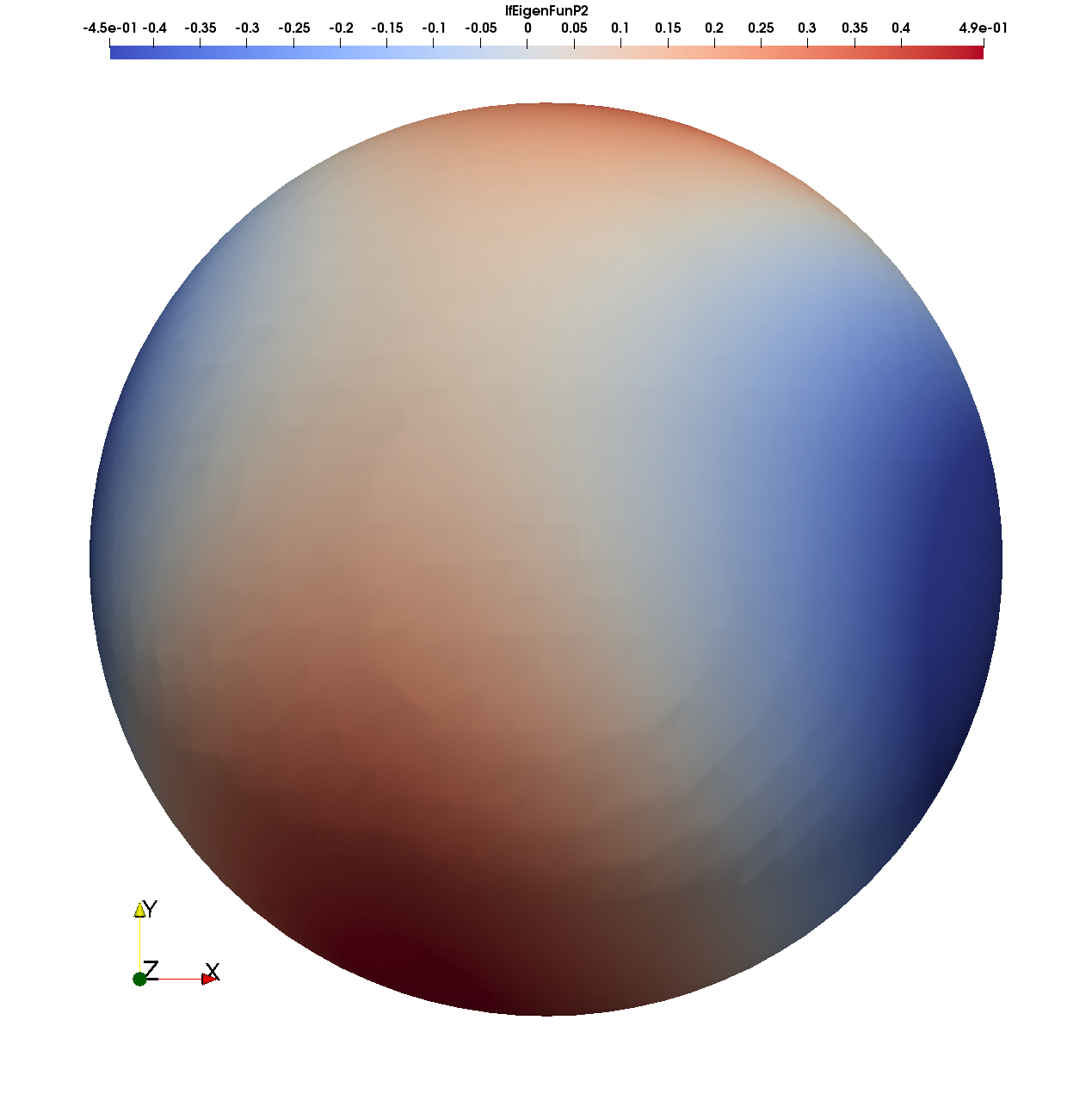}}
\hspace{0.01\linewidth}
\subfigure[$\lambda_{h,6} \approx 6 
$]{
\includegraphics[width=0.23\linewidth]{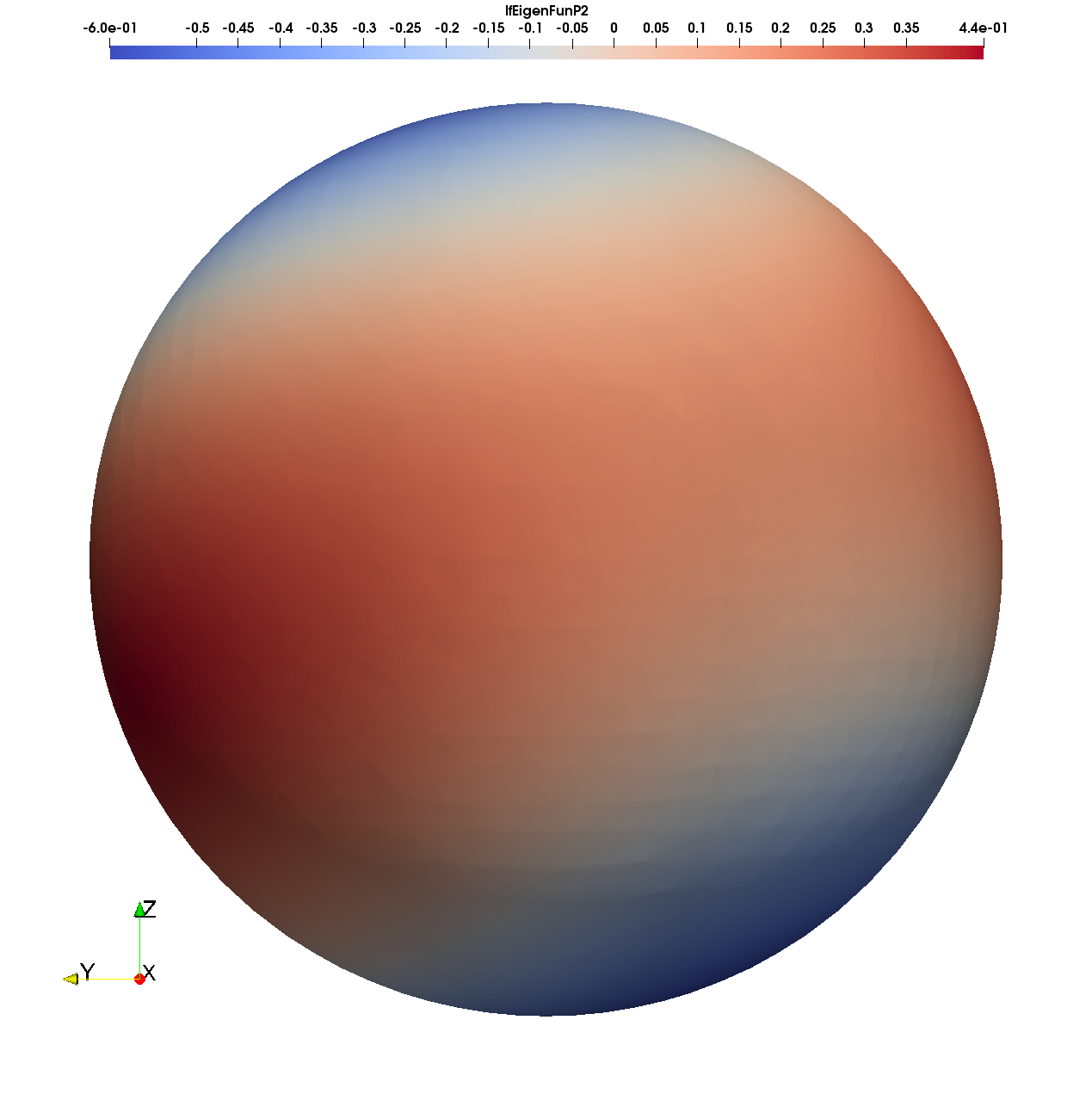}}
\subfigure[$\lambda_{h,6} \approx 6 
$]{
\includegraphics[width=0.23\linewidth]{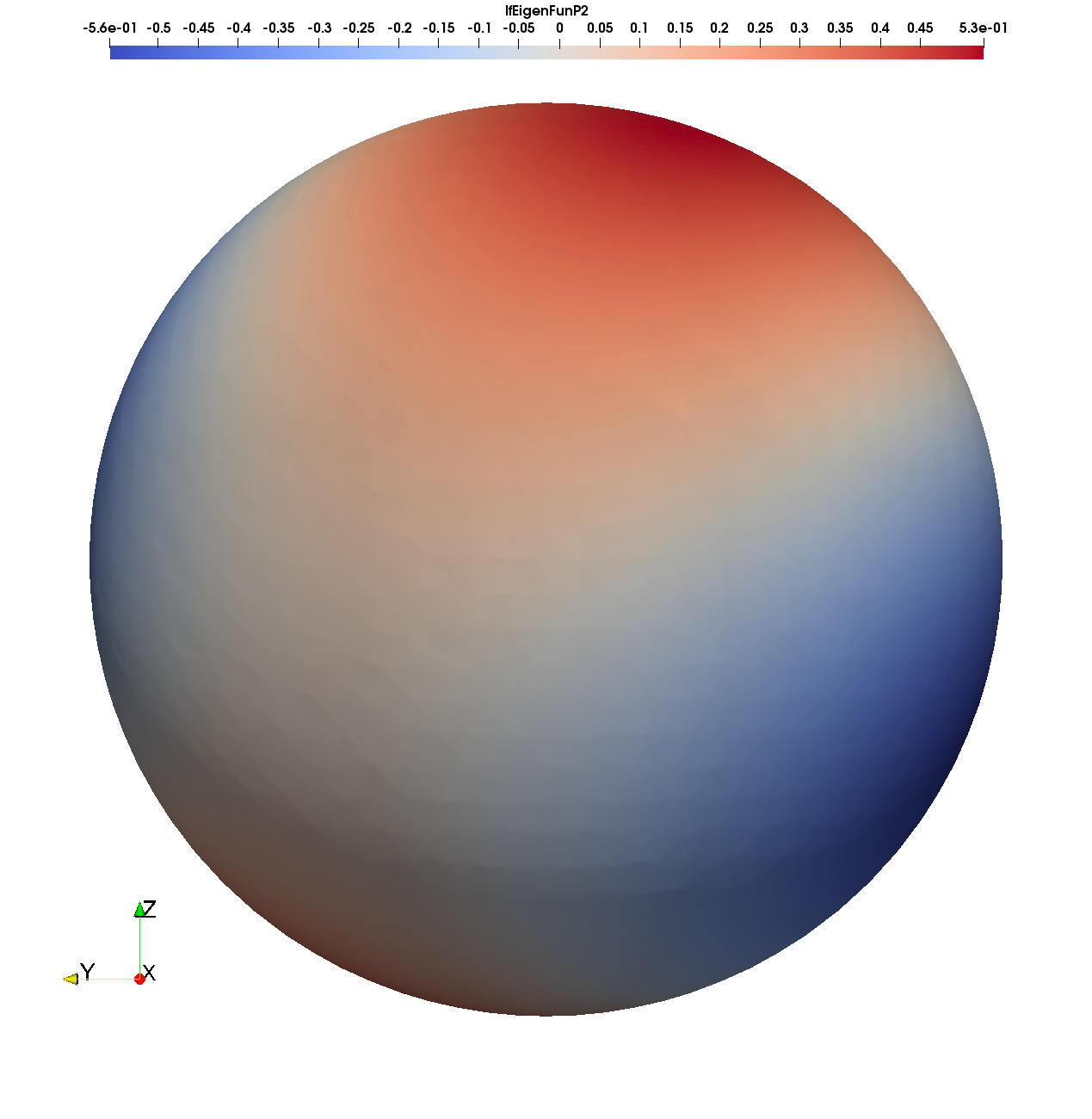}}
\hspace{0.01\linewidth}
\subfigure[$\lambda_{h,7} \approx 6 
$]{
\includegraphics[width=0.23\linewidth]{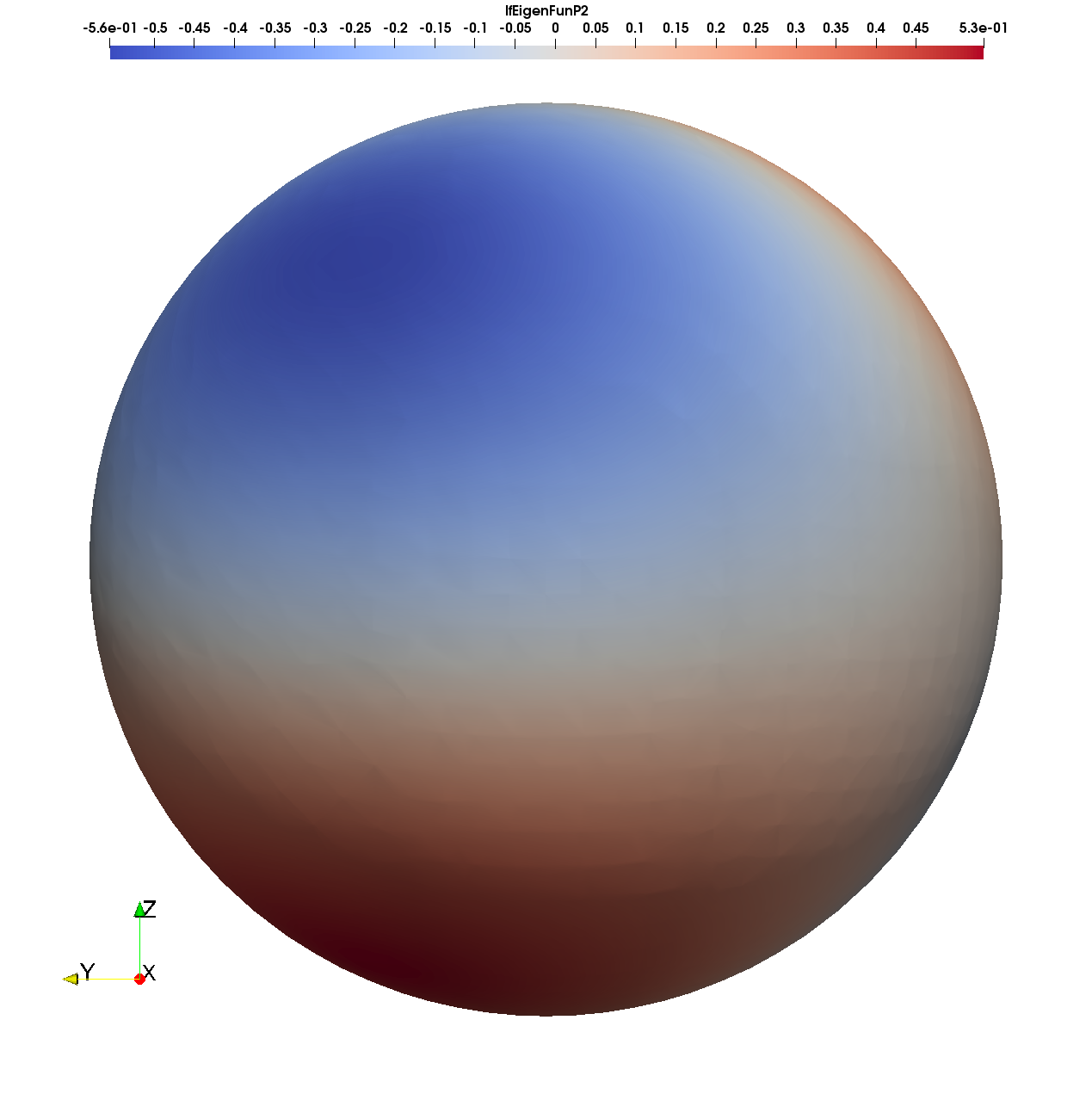}}
\hspace{0.01\linewidth}
\subfigure[$\lambda_{h,8} \approx 6 
$]{
\includegraphics[width=0.23\linewidth]{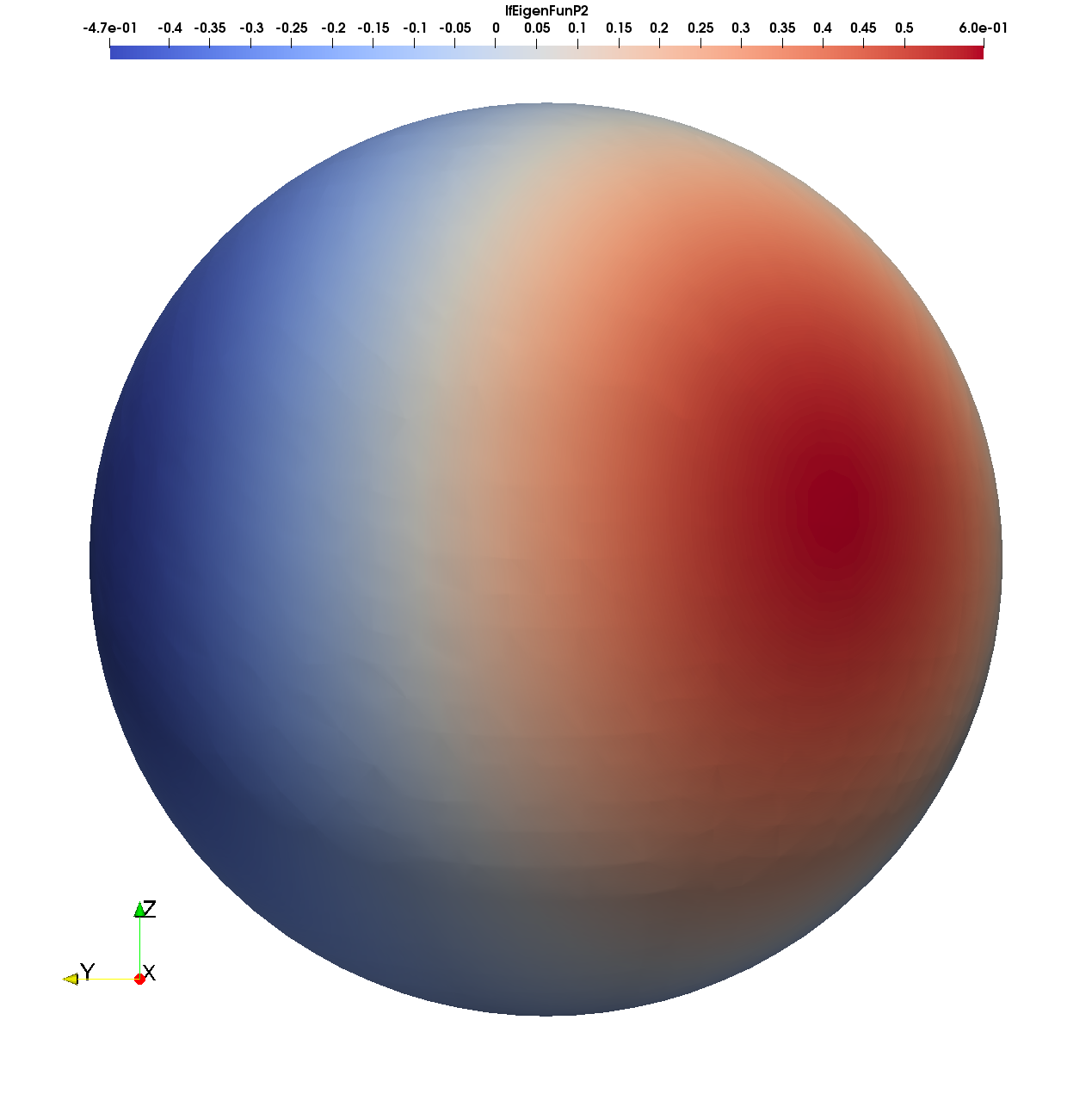}}
\caption{The eigenfunctions for the \m{Laplace-Beltrami} operator on the unit spherical surface.}\label{eig1to6}
\end{figure}


{\it Example 3.} In the last example, we solve the Laplace-Beltrami eigenvalue problem on a more general surface. Consider \m{a tooth-shaped} surface which is the zero  level set of a function
\begin{equation}
\varphi(x,y, z)=\frac{256\,x^4}{625}-\frac{16\,x^2}{25}+\frac{256\,y^4}{625}-\frac{16\,y^2}{25}+\frac{256\,z^4}{625}-\frac{16\,z^2}{25}.
\end{equation}
By using our method, the set of true eigenvalues and the corresponding eigenfunctions on the surface are shown in Figure~\ref{tooth eigenvalues}($k=1,N=32$). We show only the first six ones  for simplicity. We can see that our method works well for this problem.
\begin{figure}
\centering
\subfigure[$\lambda_{h,1} = 9.0291211439e-12$]{
\includegraphics[width=0.34\linewidth]{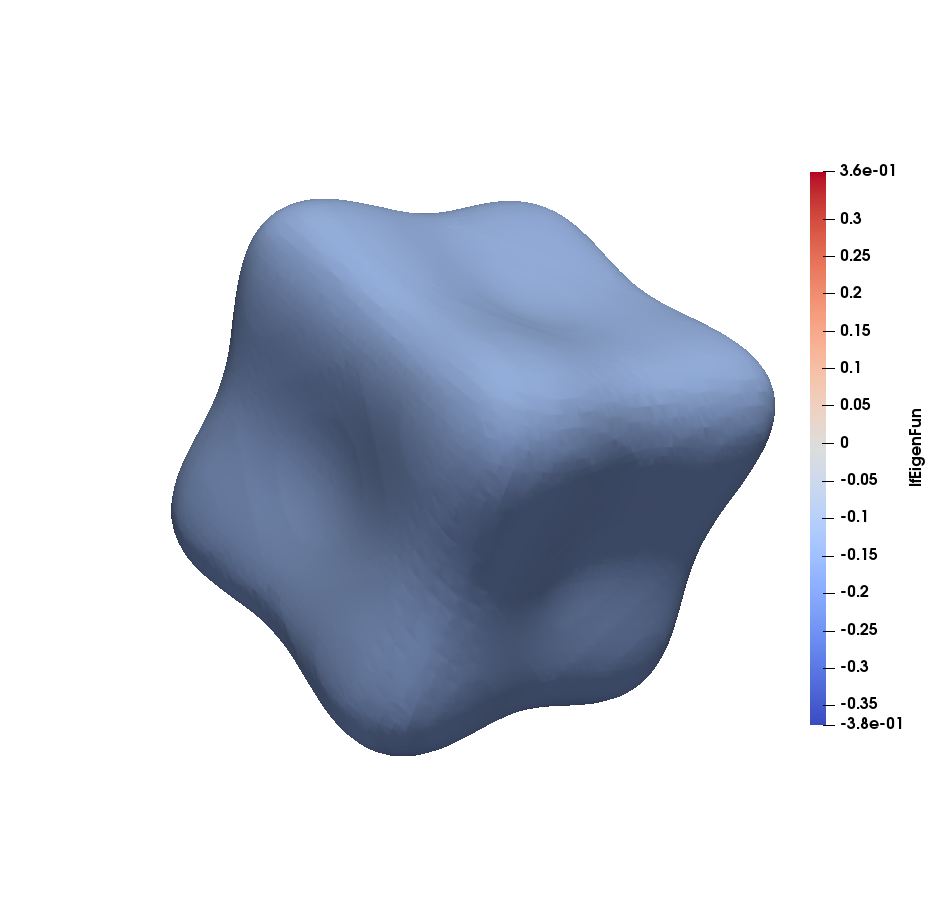}}
\hspace{0.01\linewidth}
\subfigure[$\lambda_{h,2}  = 0.5915189071$]{
\includegraphics[width=0.3\linewidth]{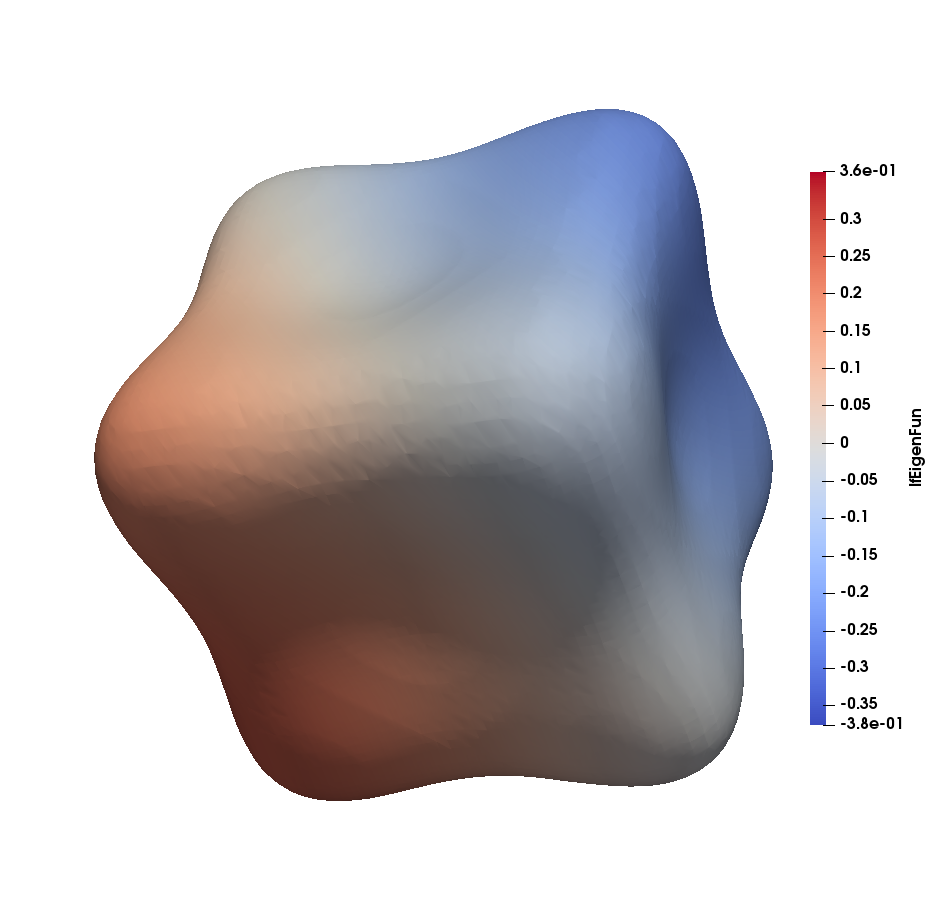}}
\hspace{0.01\linewidth}
\subfigure[$\lambda_{h,3}  = 0.5915189381$]{
\includegraphics[width=0.3\linewidth]{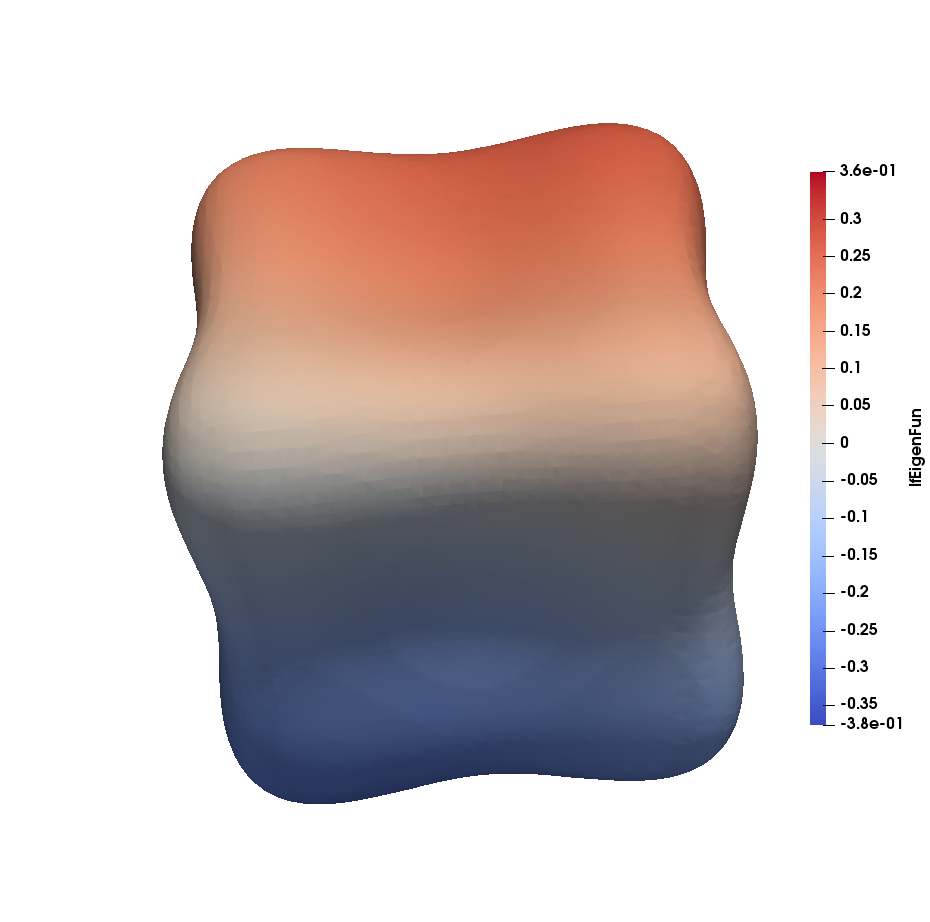}}
\vfill
\subfigure[$\lambda_{h,4}  = 0.5915934027$]{
\includegraphics[width=0.3\linewidth]{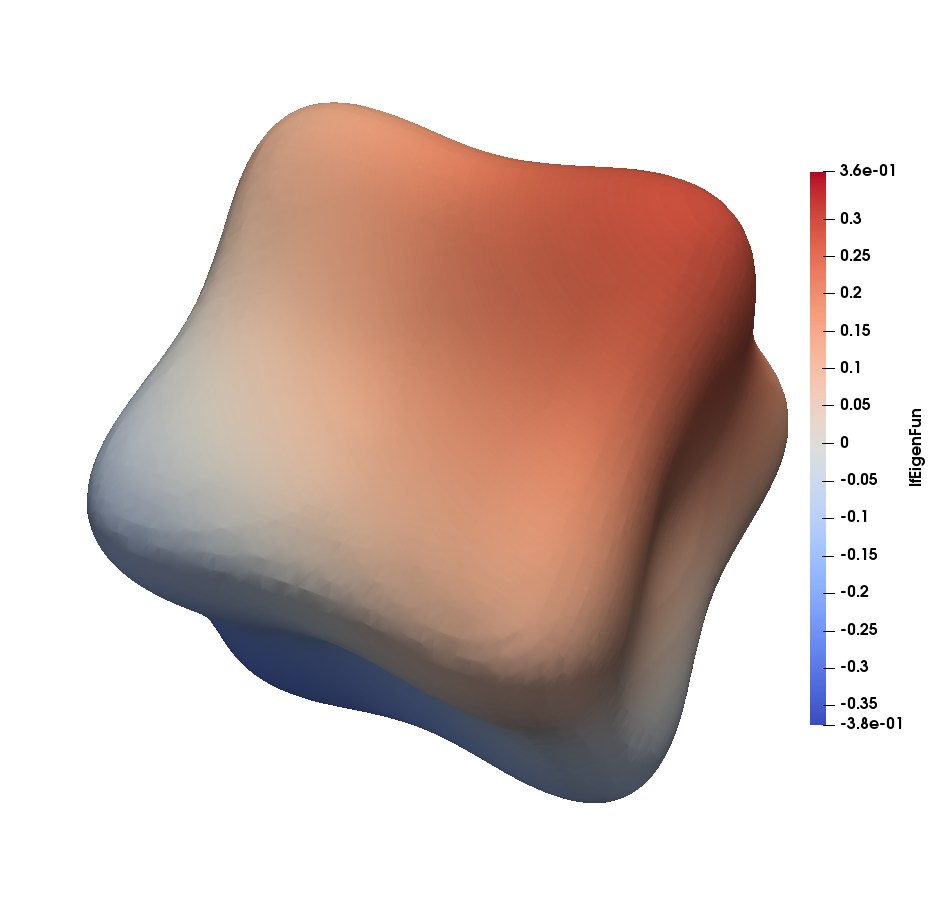}}
\hspace{0.01\linewidth}
\subfigure[$\lambda_{h,5}  = 1.612570362$]{
\includegraphics[width=0.3\linewidth]{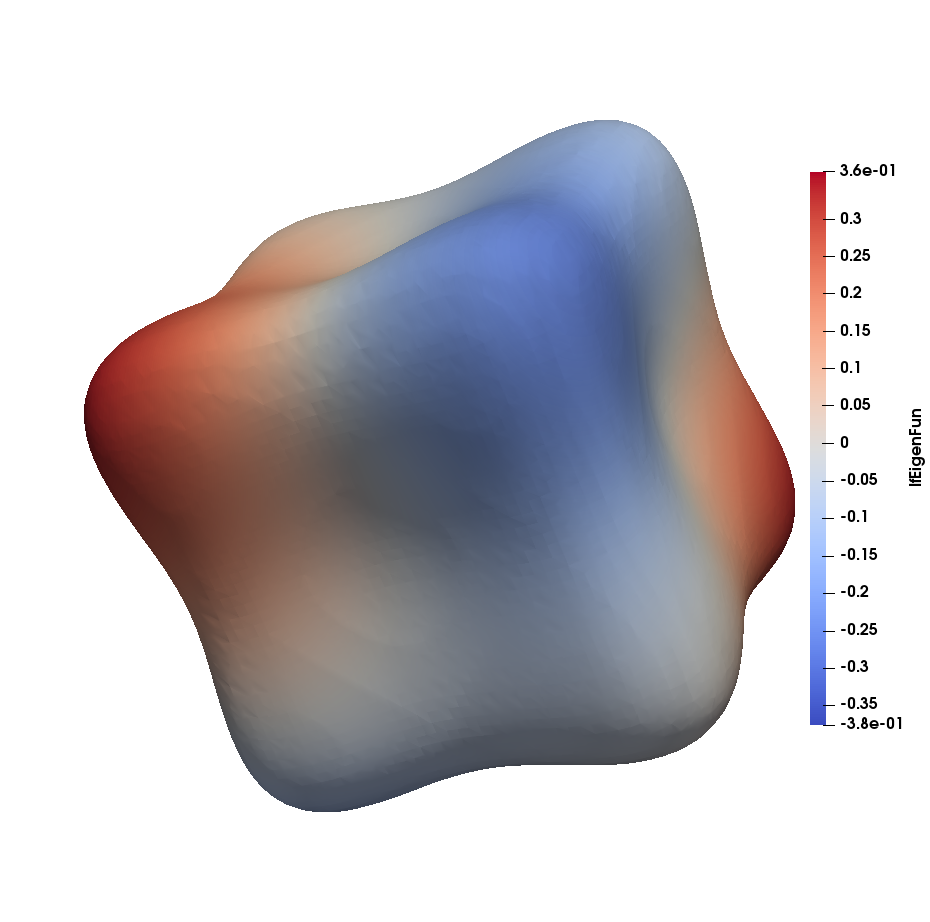}}
\hspace{0.01\linewidth}
\subfigure[$\lambda_{h,6}  = 1.612570431$]{
\includegraphics[width=0.3\linewidth]{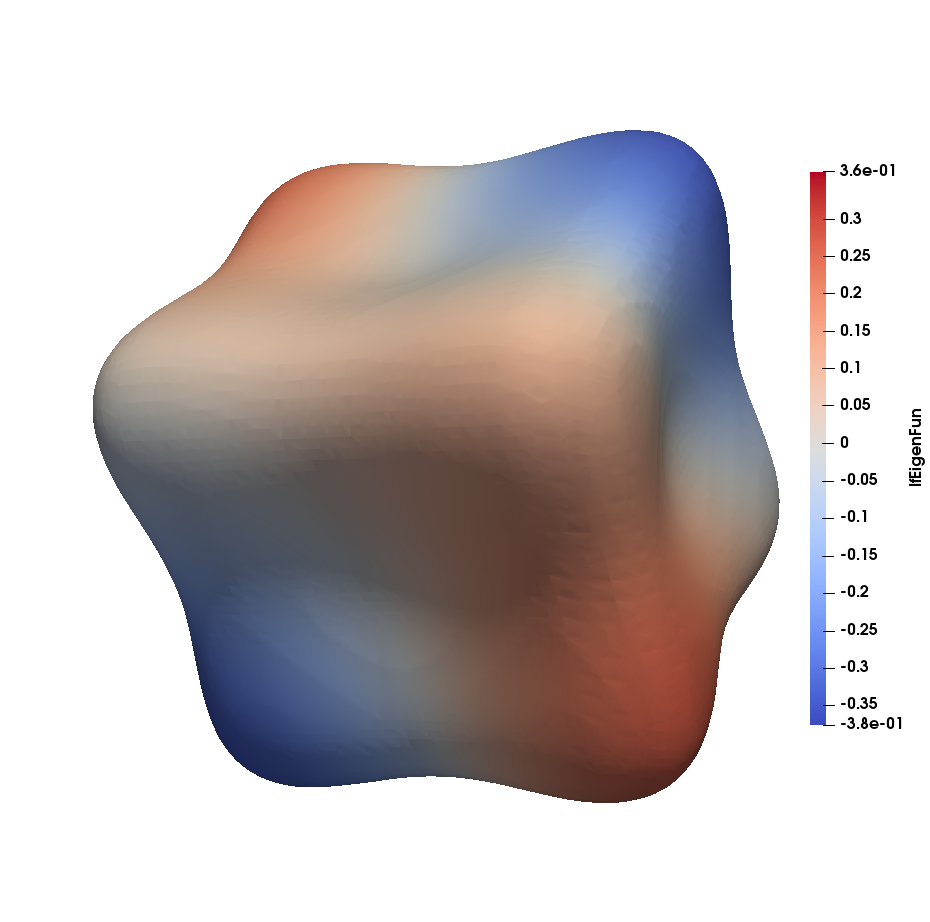}}
\caption{The Laplace-Beltrami eigenvalues and eigenfunctions on a general surface.}\label{tooth eigenvalues}
\end{figure}

\section{Conclusion}
In this paper, we develop a new trace finite element method for partial differential equations on the smooth surfaces. There
is no geometric error in the discretization of the surface since numerical integration is done directly  on the surfaces. We focus on the  Laplace-Beltrami eigenvalue problems, which \m{have not} been studied by the trace finite element
methods. The difficulties arise from the fact that the dimension of the finite element space is usually smaller than the degree of {freedoms}. Therefore, the direct application of the trace FEM to the eigenvalue
problem will lead to many false eigenvalues.

We analyse carefully the eigenvalues of the
discrete problems of the trace FEM and show that the geometric consistency property of the trace FEMs
 can improve the situation largely. Theoretically, it is possible that the discrete system of our method \m{does not} have any false eigenvalues if the surface is not
 \m{a part of} the zero level-set of some finite element function defined in the bulk domain. This is also verified by our numerical experiments. Furthermore, for the corresponding generalized matrix eigenvalue problem with false eigenvalues,
we show that a perturbed method can be used to compute the true eigenvalues. It turns out that
the algorithm \m{works well} for both the standard trace finite element method and
the geometric consistent method.
Our experiments also show that the geometric
consistent method \m{improves} the accuracy dramatically in comparison with the traditional methods. We present a numerical analysis of our method for both the Laplace-Beltrami equation and the corresponding eigenvalue problem. In comparison \m{to the methods} defined on the discrete surface, the analysis becomes much simpler.  Our method can also be easily extended to any higher order finite element methods by considering proper higher order numerical quadrature on the smooth surface.

In this work, we mainly \m{focus on} the approximation \m{property} of the method. It is known that the condition number of the trace finite element method might be large due to the irregular surface mesh induced by the bulk triangulations \cite{olshanskii2010finite,olshanskii2012surface}. Usually some stabilization terms can be added to the method to improve the numerical stability \cite{burman2015stabilized,burman2016full,larson2020stabilization,lehrenfeld2018stabilized}. These techniques can be used in our method as well and it might become necessary for higher order methods. This will be left for future work.
\bibliographystyle{siam}
\bibliography{highquad}

\begin{thebibliography}{10}

\bibitem{arroyo2009relaxation}
{\sc M.~Arroyo and A.~DeSimone}, {\em Relaxation dynamics of fluid membranes},
  Physical review E, 79 (2009), p.~031915.

\bibitem{babuvska1991eigenvalue}
{\sc I.~Babu{\v{s}}ka and J.~Osborn}, {\em Eigenvalue problems},  (1991).

\bibitem{bachini2021intrinsic}
{\sc E.~Bachini, M.~W. Farthing, and M.~Putti}, {\em Intrinsic finite element
  method for advection-diffusion-reaction equations on surfaces}, Journal of
  Computational Physics, 424 (2021), p.~109827.

\bibitem{beale2020solving}
{\sc J.~T. Beale}, {\em Solving partial differential equations on closed
  surfaces with planar cartesian grids}, SIAM Journal on Scientific Computing,
  42 (2020), pp.~A1052--A1070.

\bibitem{bertalmio2000variational}
{\sc M.~Bertalmio, L.-T. Cheng, S.~Osher, and S.~Guillermo}, {\em Variational
  problems and partial differential equations on implicit surfaces: The
  framework and examples in image processing and pattern formation},  (2000).

\bibitem{bertalmio2000framework}
{\sc M.~Bertalmio, G.~Sapiro, L.-T. Cheng, and S.~Osher}, {\em A framework for
  solving surface partial differential equations for computer graphics
  applications}, CAM Report 00-43, UCLA, Mathematics Department, 3 (2000).

\bibitem{boffi2010finite}
{\sc D.~Boffi}, {\em Finite element approximation of eigenvalue problems}, Acta
  Numerica, 19 (2010), pp.~1--120.

\bibitem{bonito2020finite}
{\sc A.~Bonito, A.~Demlow, and R.~H. Nochetto}, {\em Finite element methods for
  the laplace–beltrami operator}, arXiv: Numerical Analysis, 21 (2020),
  pp.~1--103.

\bibitem{burman2015stabilized}
{\sc E.~Burman, P.~Hansbo, and M.~G. Larson}, {\em A stabilized cut finite
  element method for partial differential equations on surfaces: the
  laplace--beltrami operator}, Computer Methods in Applied Mechanics and
  Engineering, 285 (2015), pp.~188--207.

\bibitem{burman2016cutb}
{\sc E.~Burman, P.~Hansbo, M.~G. Larson, and A.~Massing}, {\em A cut
  discontinuous {G}alerkin method for the {L}aplace--{B}eltrami operator}, IMA
  Journal of Numerical Analysis, 37 (2017), pp.~138--169.

\bibitem{burman2016full}
{\sc E.~Burman, P.~Hansbo, M.~G. Larson, A.~Massing, and S.~Zahedi}, {\em Full
  gradient stabilized cut finite element methods for surface partial
  differential equations}, Computer Methods in Applied Mechanics and
  Engineering, 310 (2016), pp.~278--296.

\bibitem{buser2010geometry}
{\sc P.~Buser}, {\em Geometry and spectra of compact Riemann surfaces},
  Springer Science \& Business Media, 2010.

\bibitem{chernyshenko2015adaptive}
{\sc A.~Y. Chernyshenko and M.~A. Olshanskii}, {\em An adaptive octree finite
  element method for {PDEs} posed on surfaces}, Computer Methods in Applied
  Mechanics and Engineering, 291 (2015), pp.~146--172.

\bibitem{ciarlet2002finite}
{\sc P.~G. Ciarlet}, {\em The finite element method for elliptic problems},
  SIAM, 2002.

\bibitem{craioveanu2013old}
{\sc M.-E. Craioveanu, M.~Puta, and T.~RASSIAS}, {\em Old and new aspects in
  spectral geometry}, vol.~534, Springer Science \& Business Media, 2013.

\bibitem{cui2020high}
{\sc T.~Cui, W.~Leng, H.~Liu, L.~Zhang, and W.~Zheng}, {\em High-order
  numerical quadratures in a tetrahedron with an implicitly defined curved
  interface}, ACM Transactions on Mathematical Software, 46 (2020), pp.~1--18.

\bibitem{deckelnick2010h}
{\sc K.~Deckelnick, G.~Dziuk, C.~M. Elliott, and C.-J. Heine}, {\em An h-narrow
  band finite-element method for elliptic equations on implicit surfaces}, IMA
  Journal of Numerical Analysis, 30 (2010), pp.~351--376.

\bibitem{dede2015isogeometric}
{\sc L.~Ded{\`e} and A.~Quarteroni}, {\em Isogeometric analysis for second
  order partial differential equations on surfaces}, Computer Methods in
  Applied Mechanics and Engineering, 284 (2015), pp.~807--834.

\bibitem{demlow2009higher}
{\sc A.~Demlow}, {\em Higher-order finite element methods and pointwise error
  estimates for elliptic problems on surfaces}, SIAM Journal on Numerical
  Analysis, 47 (2009), pp.~805--827.

\bibitem{demlow2007adaptive}
{\sc A.~Demlow and G.~Dziuk}, {\em An adaptive finite element method for the
  laplace--beltrami operator on implicitly defined surfaces}, SIAM Journal on
  Numerical Analysis, 45 (2007), pp.~421--442.

\bibitem{Demmel2000}
{\sc J.~Demmel}, {\em Generalized non-hermitian eigenproblems}, in Templates
  for the solution of algebraic eigenvalue problems: a practical guide, edit by
  Bai, Zhaojun and Demmel, James and Dongarra, Jack and Ruhe, Axel and van der
  Vorst, Henk, pp.~28--36.

\bibitem{demmel1993generalized}
{\sc J.~Demmel and B.~K{\aa}gstr{\"o}m}, {\em The generalized schur
  decomposition of an arbitrary pencil a--$\lambda$b—robust software with
  error bounds and applications. part i: theory and algorithms}, ACM
  Transactions on Mathematical Software (TOMS), 19 (1993), pp.~160--174.

\bibitem{dong2020discontinuous}
{\sc G.~Dong, H.~Guo, and Z.~Shi}, {\em Discontinuous galerkin methods for the
  laplace-beltrami operator on point cloud}, arXiv preprint arXiv:2012.15433,
  (2020).

\bibitem{DROPS}
{\em {DROPS package}}.
\newblock \verb|http://www.igpm.rwth-aachen.de/DROPS/|.

\bibitem{du2003voronoi}
{\sc Q.~Du, M.~D. Gunzburger, and L.~Ju}, {\em Voronoi-based finite volume
  methods, optimal voronoi meshes, and pdes on the sphere}, Computer methods in
  applied mechanics and engineering, 192 (2003), pp.~3933--3957.

\bibitem{dziuk1988finite}
{\sc G.~Dziuk}, {\em Finite elements for the beltrami operator on arbitrary
  surfaces},  (1988).

\bibitem{dziuk2007finite}
{\sc G.~Dziuk and C.~M. Elliott}, {\em Finite elements on evolving surfaces},
  IMA journal of numerical analysis, 27 (2007), pp.~262--292.

\bibitem{dziuk2007surface}
\leavevmode\vrule height 2pt depth -1.6pt width 23pt, {\em Surface finite
  elements for parabolic equations}, Journal of Computational Mathematics,
  (2007), pp.~385--407.

\bibitem{dziuk2013finite}
\leavevmode\vrule height 2pt depth -1.6pt width 23pt, {\em Finite element
  methods for surface pdes}, Acta Numerica, 22 (2013), p.~289.

\bibitem{ElliotStinner}
{\sc C.~M. Elliott and B.~Stinner}, {\em Modeling and computation of two phase
  geometric biomembranes using surface finite elements}, Journal of
  Computational Physics, 226 (2007), pp.~1271--1290.

\bibitem{gfrerer2018high}
{\sc M.~H. Gfrerer and M.~Schanz}, {\em A high-order fem with exact geometry
  description for the laplacian on implicitly defined surfaces}, International
  Journal for Numerical Methods in Engineering, 114 (2018), pp.~1163--1178.

\bibitem{glowinski2008computing}
{\sc R.~Glowinski and D.~C. Sorensen}, {\em Computing the eigenvalues of the
  laplace-beltrami operator on the surface of a torus: A numerical approach},
  in Partial differential equations, Springer, 2008, pp.~225--232.

\bibitem{gordon1992isospectral}
{\sc C.~Gordon, D.~Webb, and S.~Wolpert}, {\em Isospectral plane domains and
  surfaces via riemannian orbifolds}, Inventiones mathematicae, 110 (1992),
  pp.~1--22.

\bibitem{gordon1992one}
{\sc C.~Gordon, D.~L. Webb, and S.~Wolpert}, {\em One cannot hear the shape of
  a drum}, Bulletin of the American Mathematical Society, 27 (1992),
  pp.~134--138.

\bibitem{grande2014eulerian}
{\sc J.~Grande}, {\em Eulerian finite element methods for parabolic equations
  on moving surfaces}, SIAM journal on scientific computing, 36 (2014),
  pp.~B248--B271.

\bibitem{grande2018analysis}
{\sc J.~Grande, C.~Lehrenfeld, and A.~Reusken}, {\em Analysis of a high-order
  trace finite element method for pdes on level set surfaces}, SIAM Journal on
  Numerical Analysis, 56 (2018), pp.~228--255.

\bibitem{grande2016higher}
{\sc J.~Grande and A.~Reusken}, {\em A higher order finite element method for
  partial differential equations on surfaces}, SIAM Journal on Numerical
  Analysis, 54 (2016), pp.~388--414.

\bibitem{GReusken2011}
{\sc S.~{Gro\ss} and A.~Reusken}, {\em Numerical Methods for Two-phase
  Incompressible Flows}, Springer, Berlin, 2011.

\bibitem{hebey1996sobolev}
{\sc E.~Hebey}, {\em Sobolev spaces on Riemannian manifolds}, vol.~1635,
  Springer Science \& Business Media, 1996.

\bibitem{hochstenbach2019solving}
{\sc M.~E. Hochstenbach, C.~Mehl, and B.~Plestenjak}, {\em Solving singular
  generalized eigenvalue problems by a rank-completing perturbation}, SIAM
  Journal on Matrix Analysis and Applications, 40 (2019), pp.~1022--1046.

\bibitem{kovacs2018high}
{\sc B.~Kovács}, {\em High-order evolving surface finite element method for
  parabolic problems on evolving surfaces}, IMA Journal of Numerical Analysis,
  38 (2018), pp.~430--459.

\bibitem{larson2020stabilization}
{\sc M.~G. Larson and S.~Zahedi}, {\em Stabilization of high order cut finite
  element methods on surfaces}, IMA Journal of Numerical Analysis, 40 (2020),
  pp.~1702--1745.

\bibitem{lehrenfeld2016high}
{\sc C.~Lehrenfeld}, {\em High order unfitted finite element methods on level
  set domains using isoparametric mappings}, Computer Methods in Applied
  Mechanics and Engineering, 300 (2016), pp.~716--733.

\bibitem{lehrenfeld2018stabilized}
{\sc C.~Lehrenfeld, M.~A. Olshanskii, and X.~Xu}, {\em A stabilized trace
  finite element method for partial differential equations on evolving
  surfaces}, SIAM Journal on Numerical Analysis, 56 (2018), pp.~1643--1672.

\bibitem{lehto2017radial}
{\sc E.~Lehto, V.~Shankar, and G.~B. Wright}, {\em A radial basis function
  (rbf) compact finite difference (fd) scheme for reaction-diffusion equations
  on surfaces}, SIAM Journal on Scientific Computing, 39 (2017),
  pp.~A2129--A2151.

\bibitem{leung2011grid}
{\sc S.~Leung, J.~Lowengrub, and H.~Zhao}, {\em A grid based particle method
  for solving partial differential equations on evolving surfaces and modeling
  high order geometrical motion}, Journal of Computational Physics, 230 (2011),
  pp.~2540--2561.

\bibitem{li2016convergent}
{\sc Z.~Li and Z.~Shi}, {\em A convergent point integral method for isotropic
  elliptic equations on a point cloud}, Multiscale Modeling \& Simulation, 14
  (2016), pp.~874--905.

\bibitem{liang2013solving}
{\sc J.~Liang and H.~Zhao}, {\em Solving partial differential equations on
  point clouds}, SIAM Journal on Scientific Computing, 35 (2013),
  pp.~A1461--A1486.

\bibitem{macdonald2011solving}
{\sc C.~B. Macdonald, J.~Brandman, and S.~J. Ruuth}, {\em Solving eigenvalue
  problems on curved surfaces using the closest point method}, Journal of
  Computational Physics, 230 (2011), pp.~7944--7956.

\bibitem{mckean1967curvature}
{\sc H.~P. McKean~Jr and I.~M. Singer}, {\em Curvature and the eigenvalues of
  the laplacian}, Journal of Differential Geometry, 1 (1967), pp.~43--69.

\bibitem{Milliken}
{\sc W.~Milliken, H.~Stone, and L.~Leal}, {\em The effect of surfactant on
  transient motion of newtonian drops}, Phys. Fluids A, 5 (1993), pp.~69--79.

\bibitem{muhivc2009singular}
{\sc A.~Muhi{\v{c}} and B.~Plestenjak}, {\em On the singular two-parameter
  eigenvalue problem}, The Electronic Journal of Linear Algebra, 18 (2009),
  pp.~420--437.

\bibitem{muller2013highly}
{\sc B.~M{\"u}ller, F.~Kummer, and M.~Oberlack}, {\em Highly accurate surface
  and volume integration on implicit domains by means of moment-fitting},
  International Journal for Numerical Methods in Engineering, 96 (2013),
  pp.~512--528.

\bibitem{nasikun2018fast}
{\sc A.~Nasikun, C.~Brandt, and K.~Hildebrandt}, {\em Fast approximation of
  laplace-beltrami eigenproblems}, in Computer Graphics Forum, vol.~37, Wiley
  Online Library, 2018, pp.~121--134.

\bibitem{Novaketal}
{\sc I.~L. Novak, F.~Gao, Y.-S. Choi, D.~Resasco, J.~C. Schaff, and
  B.~Slepchenko}, {\em Diffusion on a curved surface coupled to diffusion in
  the volume: application to cell biology}, Journal of Computational Physics,
  229 (2010), pp.~6585--6612.

\bibitem{olshanskii2021inf}
{\sc M.~Olshanskii, A.~Reusken, and A.~Zhiliakov}, {\em Inf-sup stability of
  the trace p2-p1 taylor--hood elements for surface pdes}, Mathematics of
  Computation, 90 (2021), pp.~1527--1555.

\bibitem{olshanskii2021finite}
{\sc M.~Olshanskii, X.~Xu, and V.~Yushutin}, {\em A finite element method for
  allen--cahn equation on deforming surface}, Computers \& Mathematics with
  Applications, 90 (2021), pp.~148--158.

\bibitem{olshanskii2010finite}
{\sc M.~A. Olshanskii and A.~Reusken}, {\em A finite element method for surface
  pdes: matrix properties}, Numerische Mathematik, 114 (2010), p.~491.

\bibitem{olshanskii2017trace}
\leavevmode\vrule height 2pt depth -1.6pt width 23pt, {\em Trace finite element
  methods for pdes on surfaces}, in Geometrically unfitted finite element
  methods and applications, Springer, 2017, pp.~211--258.

\bibitem{olshanskii2009finite}
{\sc M.~A. Olshanskii, A.~Reusken, and J.~Grande}, {\em A finite element method
  for elliptic equations on surfaces}, SIAM Journal on Numerical Analysis, 47
  (2009), pp.~3339--3358.

\bibitem{olshanskii2012surface}
{\sc M.~A. Olshanskii, A.~Reusken, and X.~Xu}, {\em On surface meshes induced
  by level set functions}, Computing and visualization in science, 15 (2012),
  pp.~53--60.

\bibitem{olshanskii2014eulerian}
\leavevmode\vrule height 2pt depth -1.6pt width 23pt, {\em An {Eulerian}
  space--time finite element method for diffusion problems on evolving
  surfaces}, SIAM Journal on Numerical Analysis, 52 (2014), pp.~1354--1377.

\bibitem{PHG}
{\em {PHG package}}.
\newblock \verb|http://lsec.cc.ac.cn/phg/|.

\bibitem{reusken2015analysis}
{\sc A.~Reusken}, {\em Analysis of trace finite element methods for surface
  partial differential equations}, IMA Journal of Numerical Analysis, 35
  (2015), pp.~1568--1590.

\bibitem{reuter2006laplace}
{\sc M.~Reuter, F.-E. Wolter, and N.~Peinecke}, {\em Laplace--beltrami spectra
  as ‘shape-dna’of surfaces and solids}, Computer-Aided Design, 38 (2006),
  pp.~342--366.

\bibitem{ruuth2008simple}
{\sc S.~J. Ruuth and B.~Merriman}, {\em A simple embedding method for solving
  partial differential equations on surfaces}, Journal of Computational
  Physics, 227 (2008), pp.~1943--1961.

\bibitem{saye2015high}
{\sc R.~Saye}, {\em High-order quadrature methods for implicitly defined
  surfaces and volumes in hyperrectangles}, SIAM Journal on Scientific
  Computing, 37 (2015), pp.~A993--A1019.

\bibitem{simons1997functional}
{\sc K.~Simons and E.~Ikonen}, {\em Functional rafts in cell membranes},
  Nature, 387 (1997), p.~569.

\bibitem{Stone}
{\sc H.~Stone}, {\em A simple derivation of the time-dependent
  convective-diffusion equation for surfactant transport along a deforming
  interface}, Phys. Fluids A, 2 (1990), pp.~111--112.

\bibitem{sun2016finite}
{\sc J.~Sun and A.~Zhou}, {\em Finite element methods for eigenvalue problems},
  Chapman and Hall/CRC, 2016.

\bibitem{van1979computation}
{\sc P.~Van~Dooren}, {\em The computation of kronecker's canonical form of a
  singular pencil}, Linear Algebra and Its Applications, 27 (1979),
  pp.~103--140.

\bibitem{weyl1912asymptotische}
{\sc H.~Weyl}, {\em Das asymptotische verteilungsgesetz der eigenwerte linearer
  partieller differentialgleichungen (mit einer anwendung auf die theorie der
  hohlraumstrahlung)}, Mathematische Annalen, 71 (1912), pp.~441--479.

\bibitem{wilkinson1979kronecker}
{\sc J.~H. Wilkinson}, {\em Kronecker's canonical form and the qz algorithm},
  Linear Algebra and its Applications, 28 (1979), pp.~285--303.

\bibitem{xu2004discrete}
{\sc G.~Xu}, {\em Discrete laplace--beltrami operators and their convergence},
  Computer aided geometric design, 21 (2004), pp.~767--784.

\bibitem{XuZh}
{\sc J.-J. Xu and H.-K. Zhao}, {\em An {Eulerian} formulation for solving
  partial differential equations along a moving interface}, Journal of
  Scientific Computing, 19 (2003), pp.~573--594.

\bibitem{yushutin2020numerical}
{\sc V.~Yushutin, A.~Quaini, and M.~Olshanskii}, {\em Numerical modeling of
  phase separation on dynamic surfaces}, Journal of Computational Physics, 407
  (2020), p.~109126.

\end{thebibliography}
\end{document}